    \documentclass[11pt,a4paper,reqno]{amsart}

    \usepackage{bm}
    \usepackage{graphicx}
    \usepackage[utf8]{inputenc}
    \usepackage[T1]{fontenc}
    \usepackage[margin=1in]{geometry}
    \usepackage{amsmath}
    \usepackage{amsthm}
    \usepackage{amssymb}
    \usepackage{amsopn}
    \usepackage{enumitem}
    \usepackage{xcolor}
    \usepackage[group-separator={,}]{siunitx}
    
    \usepackage{mdframed}

    \usepackage{mathtools}
    \usepackage{hyperref}
    \usepackage{doi}
    \usepackage{xfrac}
    \usepackage{graphicx}
    \usepackage[export]{adjustbox}
    \usepackage{bbm}
    \usepackage{relsize}
    
    \usepackage{algpseudocode}
    \usepackage{algorithm}

    \usepackage[tableposition=top]{caption}
    \usepackage[labelformat=simple]{subcaption}
    
    \usepackage[normalem]{ulem}
    \usepackage{bm}
    
    \theoremstyle{plain}
    \newtheorem{theorem}{Theorem}[section]
    \newtheorem{prop}[theorem]{Proposition}
    
    \newtheorem{lemma}[theorem]{Lemma}
    
    \theoremstyle{definition}
    \newtheorem{definition}[theorem]{Definition}

    \theoremstyle{remark}
    \newtheorem{remark}[theorem]{Remark}
    
    \numberwithin{equation}{section}

    \newcommand{\La}{\Lambda}
    \newcommand{\R}{\mathbb{R}}
    \newcommand{\Rc}{\ensuremath{\mathcal{R}}}
    \newcommand{\Z}{\mathbb{Z}}

    \newcommand{\E}{\mathcal{E}}

    \newcommand{\Hcc}{\ensuremath{\dot{\mathcal H}^1}}
    \newcommand{\Hc}{\ensuremath{{\mathcal H}_c}}
    
    \newcommand{\abs}[1]{\lvert#1\rvert}
    
    \DeclareMathOperator{\sgn}{sgn}
    \DeclareMathOperator{\divo}{div}

    \def\XXint#1#2#3{{\setbox0=\hbox{$#1{#2#3}{\int}$ }
    \vcenter{\hbox{$#2#3$ }}\kern-.6\wd0}}

    \makeatletter
    \def\paragraph{\@startsection{paragraph}{4}%
      \z@\z@{-\fontdimen2\font}%
      {\normalfont\bfseries}}
    \makeatother
    
    \makeatletter
    \def\subsubsection{\@startsection{subsubsection}{3}%
      \z@{.5\linespacing\@plus.7\linespacing}{-.5em}%
      {\normalfont\bfseries}}
      \makeatother

    \usepackage{color}
    
    \graphicspath{{./graphics/}}
    
\begin{document}
    \title[Incompleteness of continuum flex boundary conditions for atomistic fracture]{Incompleteness of Sinclair-type continuum flexible boundary conditions for atomistic fracture simulations}
    \author{Julian Braun}
    \address{Maxwell Institute for Mathematical Sciences and Department of Mathematics, Heriot-Watt University, Edinburgh, EH14 4AS, United Kingdom}
    \email{j.braun@hw.ac.uk}
    \author{Maciej Buze}
    \email{m.buze@hw.ac.uk}
    \thanks{MB is supported by EPSRC, Grant No. EP/V00204X/1.} 
    
    \subjclass[2020]{74G10, 70C20, 74A45, 74E15, 41A58, 74G15}
    
    \date{\today}
    
    \keywords{crystalline solids, fracture, asymptotic expansion, elastic far-field}
    
    \begin{abstract}
    The elastic field around a crack opening is known to be described by continuum linearised elasticity in leading order. In this work, we explicitly develop the next term in the atomistic asymptotic expansion in the case of a Mode III crack in anti-plane geometry.
    The aim of such an expansion is twofold. First, we show that the well-known flexible boundary condition ansatz due to Sinclair is incomplete, meaning that, in principle, employing it in atomistic fracture simulations is no better than using boundary conditions from continuum linearised elasticity. And secondly, the higher order far-field expansion can be employed as a boundary condition for high-accuracy atomistic simulations.
    To obtain our results, we develop an asymptotic expansion of the associated lattice Green's function. In an interesting departure from the recently developed theory for spatially homogeneous cases, this includes a novel notion of a discrete geometry predictor, which accounts for the peculiar discrete geometry near the crack tip. 
    \end{abstract}
    
    \maketitle
    
    \section{Introduction}
    
    Fracture mechanics has long been a cornerstone in the understanding of material failure, providing essential insight into the initiation and propagation of cracks in diverse structural components \cite{lawn_1993,SJ12}. On length scales where continuum fracture mechanics ceases to accurately describe matter, atomistic modelling techniques, particularly molecular dynamics (MD) simulations, have emerged as powerful tools for investigating atomistic fracture processes \cite{Bitzek2015}. The accuracy of these simulations crucially depends on the appropriate specification of boundary conditions, which describe the state of the material and its surroundings outside of the core region of interest, e.g. the crack tip. 
    
    Traditional atomistic simulations often employ periodic boundary conditions (PBC) to mitigate finite-size effects and mimic an infinite material \cite{frenkel2002understanding}. While PBCs have proven to be very effective for systems naturally permitting a periodic setup (e.g. localised defects), their use in the modelling of fracture is limited, since a half-infinite crack breaks the translational symmetry of the system. Trying to mitigate that while retaining a periodic setup may introduce artificial constraints that limit the realism of the simulated fracture process \cite{Bitzek2015}.
    
    An alternative commonly employed is to prescribe boundary conditions from continuum fracture mechanics, either simply from linear elasticity or from an asymptotic expansion that is based on the continuum theory (the so-called flexible boundary conditions due to Sinclair and co-authors \cite{sinclair1972atomistic,sinclair1975influence,sinclair1978flexible}). Such an approach has been employed in a number of recent atomistic studies of fracture \cite{moller2014comparative,andric2018atomistic,hiremath2022effects, lakshmipathy2022lefm,zhang2023atomistic,zhang2024efficiency}.  As recognised in a growing recent mathematical literature on this topic \cite{EOS2016, 2017-bcscrew, BHO22, braun2022higher, olson2023elastic}, 
    for such an approach to be \emph{convergent}, that is to not lead to a build-up of numerical artefacts near the boundary of the computational domain, one needs to ensure that the prescribed boundary condition exhibit \emph{atomistic asymptotic consistency} up to a certain minimal order. This condition is called an \emph{approximate equilibrium} in \cite{EOS2016}.
    Moreover, prescribing boundary conditions with higher order atomistic asymptotic consistency leads to a corresponding quantifiable increase in the accuracy of the resulting numerical method.
    
    Unlike for dislocations and point defects, in the case of fracture, it follows from a simple Taylor expansion argument (see \cite[Section 3.]{2018-antiplanecrack}) that the zero order boundary condition, that is the linearised continuum elasticity boundary condition, is only convergent for simplified anti-plane mode III models \cite{2018-antiplanecrack}. It thus raises the question whether the higher-order terms in the continuum-theory-based asymptotic expansion proposed by Sinclair can lead to atomistic asymptotic consistency.  
    
    In the present work, we focus on the simplified anti-plane mode III model first analysed in \cite{2018-antiplanecrack,2019-antiplanecrack} and rigorously develop the next order in the atomistic asymptotic expansion. While the known \emph{zero order}  boundary condition (coming from continuum linearised elasticity) is given, in polar coordinates $x = r(\cos \theta, \sin \theta)$, by 
    \[
    K r^{1/2}\sin(\theta/2),
    \]
    where $K\geq 0$ is the (rescaled) stress intensity factor, we prove that the next order is given by 
    \[
    \underbrace{\vphantom{\frac{K^3}{64}} C_1(K) r^{-1/2}\sin(-\theta/2)}_{\rm (S)} \;\,\underbrace{+\,\frac{C_2 K^3}{64} r^{-1/2}\Big( \log r \sin \tfrac{\theta}{2} + \frac{1}{6} \sin\tfrac{5\theta}{2}\Big)}_{\rm (NL)},
    \]
    where $C_1(K)$ is a real constant depending on $K$ and $C_2$ is a real constant depending on the interatomic potential. The first term (S) corresponds to the aforementioned ansatz by Sinclair (see \cite[Appendix 2]{sinclair1975influence} and Remark~\ref{rem-incomplete-sinclair}), which in this case predicts the first order boundary condition to be simply of the form
    \[
    C_1(K) r^{-1/2}\sin(-\theta/2).
    \]
    We thus show that this ansatz is incomplete, as it misses the part (NL). As we will see in all detail in Section~\ref{sec:proofs-predictors} this term comes from the nonlinearity and can be seen as a solution to the corrector equation \eqref{eq:u1PDE}. The striking generic conclusion is that atomistic fracture simulations employing continuum-theory based flexible boundary conditions are only as accurate as the ones employing the zero order boundary condition. We note that such \emph{agnosticism} of continuum fracture theories to atomistic effects and nonlinearities has recently been studied numerically in \cite{lakshmipathy2022lefm}.

    A further result of our analysis here is a full characterisation of the constants $C_1$ and $C_2$. While $C_2$ is simple and explicit, $C_1$ depends non-trivially on the precise discrete behaviour around the crack tip. While the precise characterisation of $C_1$ we give here is a lot more involved, this is in line with the discrete nature of the defect dipole tensors for point defects, see \cite{BHO22, braun2022higher}.
    
    To obtain our results, we follow the atomistic asymptotic expansion framework developed in \cite{BHO22} for the spatially homogeneous case (e.g. for point defects and screw dislocations). The first step is a Taylor expansion of the energy around the reference configuration, followed by approximating finite differences with continuum differential operators and a grouping of terms according their far-field decay. This allows us to define a linear PDE to which (NL) is an explicit solution. The term (S) is an analogue of the \emph{multipole expansion} from \cite{BHO22}, however the precise structure is much more complicated, because the lattice Green's function is not spatially homogeneous. As a result, to find the right constant $C_1(K)$ in (S), a higher order expansion of the lattice Green's function in the anti-plane crack geometry is first established. In an intriguing and non-trivial departure from the homogeneous case, this requires an introduction of the novel notion of a \emph{discrete geometry predictor}, which accounts for the fact that near the crack tip, the complex square root mapping, present in the definition of zero-order Green's function predictor, distorts the lattice too much for that predictor to correctly capture the discrete effects there. This predictor is truly discrete, in the sense that it is defined as solving a linear discrete PDE. Our result in particular provides an explicit example of the extra complexity of trying to include the spatially inhomogenous cases in this theory.\\
    
    \paragraph{Outline of the paper} The paper is organised as follows. Section~\ref{sec:main} is devoted to the presentation of the main results. We introduce the atomistic model in Section~\ref{sec:model} and recall what is already known about it. Section~\ref{sec:higher_pred_u} is devoted to establishing results about the first order asymptotic expansion for the equilibrium crack displacement. In Section~\ref{sec:higher-order-G} the corresponding results for the Green's function are gathered. Section~\ref{sec:numerics} is devoted to finite-domain approximation -- we establish a convergence rate result and, based on that, present a numerical simulations confirming the veracity of the expansion. This is followed by conclusions in Section~\ref{sec:conclusions}. The proofs are gathered in the last three sections.
    
    \section{Main results}\label{sec:main}
    \subsection{The atomistic model}\label{sec:model}
    The setup closely follows the one studied in \cite{2018-antiplanecrack}. We consider the two dimensional square lattice 
    \[
    \La := \{ m - \left(\tfrac12,\tfrac12\right)\,\mid\, m \in \Z^2\},
    \]
    with a crack opening along
    \[
    \Gamma_0 := \{ (x_1,0)\,\mid\, x_1 \leq 0\}
    \]
    and the origin of the coordinate system coinciding with the crack tip. We distinguish the atoms lying on the crack surface as 
    \[
    \Gamma := \Gamma_+ \cup \Gamma_-,\quad \Gamma_{\pm} := \{ m \in \La \,\mid\, m_1 < 0, \quad m_2 = \pm \tfrac12\}.
    \]
    The atoms are assumed interact with its nearest neighbours, which, in a homogeneous lattice, gives the set of interaction stencils
    \[
    \Rc := \{e_1, e_2, -e_1, -e_2\},
    \]
    and in a cracked crystal we have, for any $m \in \La$,
    \[
    \Rc(m) := \begin{cases} \Rc\quad &\text{ for } m \not\in \Gamma,\\
    \Rc \setminus \{\mp e_2\}\quad &\text{ for } m \in \Gamma_{\pm}.
    \end{cases}
    \]
    Given an anti-plane displacement $u\,\colon\, \La \to \R$, the usual finite difference operator is \linebreak ${D_{\rho}u(m):= u(m+\rho) - u(m)}$ and, reflecting the definition of the interaction stencils, the discrete gradient $Du(m) \in \R^{\Rc}$ is given by
    \begin{equation}\label{eqn-def-Du}
    (Du(m))_{\rho} := \begin{cases}
        D_{\rho}u(m)\quad &\text{ if } \rho \in \Rc(m),\\
        0\quad &\text { if } \rho \not\in \Rc(m).
    \end{cases}
    \end{equation}
    This allows us to define the discrete Sobolev space
    \[
    \Hcc := \{ u \,\colon\, \La \to \R\,\mid\, Du \in \ell^2(\La) \},\quad  \|u\|_{\Hcc} := \|Du\|_{\ell^2(\La)} = \left(\sum_{m \in \La} |Du(m)|^2\right)^{1/2}
    \]
    Note that $\|\cdot\|_{\Hcc}$ defines a semi-norm. One can use equivalence based on constant shifts or specify the value at a specific atom when a norm is required.
    
    The primary object of our study is the energy difference functional $\E\,\colon\,\Hcc \to \R$ given by
    \begin{equation}\label{eqn-E-def}
    \E(u) := \sum_{m \in \La} V(D\hat u_0(m) + Du(m)) - V(D\hat u_0(m)),
    \end{equation}
    where $V\, \colon \R^{\Rc} \to \R$ is an interatomic potential, which in our case takes the form
    \[
    V(Du(m) = \sum_{\rho \in \Rc} \phi\left((Du(m))_{\rho}\right),
    \]
    where $\phi \in C^k(\R)$ for $k \geq 6$ is assumed to satisfy \emph{anti-plane mirror symmetry} \cite[Section~2.2]{2017-bcscrew}. The function $\hat u_0$ in \eqref{eqn-E-def} is the continuum linearised elasticity predictor given by 
    \begin{equation}\label{eqn-def-uhat0}
    \hat u_0(m) = K \omega_2(m),
    \end{equation}
    where $K\geq 0$ is the stress intensity factor and acts as a loading parameter, and $\omega \,\colon\,\R^2\setminus \Gamma_0 \to \R^2$ is the complex square root mapping given, in polar coordinates $x = r(\cos\theta,\sin\theta)$, by
    \begin{equation}\label{eqn-omega-def}
    \omega(x) = (\omega_1(x),\omega_2(x)) := \sqrt{r}(\cos\tfrac\theta 2, \sin \tfrac \theta 2).
    \end{equation}
    The following results hold in this \emph{zero-order} case \cite[Theorem~2.1, Theorem~2.2]{2018-antiplanecrack}.
    \begin{theorem}
        The energy difference $\E$ is well-defined on $\Hcc$ and $k$-times continuously differentiable. 
    \end{theorem}
    \begin{theorem}\label{thm:ubar0}
        If $K \geq 0$ is small enough, then there exist a locally unique $\bar u_0 \in \Hcc$ minimiser of $\E$ which depends linearly on $K$ and is strongly stable, that is there exists $\lambda > 0$ such that, for all $v \in \Hcc$,
        \[
        \delta^2 \E(\bar u_0)[v,v] \geq \lambda \|v\|^2_{\Hcc}.
        \]
        Furthermore, $\bar u_0$ satisfies 
        \begin{equation}\label{eqn-Dbaru0-K}
        |D\bar u_0(\ell)| \lesssim C(K) |\ell|^{-3/2 + \delta},
        \end{equation}
        for an arbitrarily small $\delta > 0$ and a constant $C$ that depends on $K$.
    \end{theorem}
    \begin{remark}
        We note that while it is claimed in \cite{2018-antiplanecrack} that the dependence in $K$ is linear, this only applies to the small-loading regime, i.e. when $K$ is assumed to be small enough.
    \end{remark}
    
    \subsection{Higher order predictors and the improved decay of the atomistic corrector}\label{sec:higher_pred_u}
    In the present work, we succeed in developing the full equilibrium displacement field $u = \hat u_0 + \bar u_0$ from Theorem~\ref{thm:ubar0} into an expansion
    \begin{subequations}\label{eqn-u-expansion}
    \begin{align}
    u &= \hat u_0 +  \bar u_0\\ &= \hat u_0 + \hat u_1 + \bar u_1\\ &= \hat u_0 + \hat u_1 + \hat u_2 + \bar u_2,
    \end{align}
    \end{subequations}
    where the predictors $\hat u_1$ and $\hat u_2$ have explicit functional forms, satisfy, for $i=1,2$,
    \[
    |D\hat u_i(\ell)| \lesssim |\ell|^{-3/2}\log |l|,
    \]
    and, crucially, fully capture the behaviour of $u$ at this order, in the sense that the following is true. 
    \begin{theorem}\label{thm:ubar2_decay}
    The corrector $\bar u_2 \in \Hcc$, defined in \eqref{eqn-u-expansion}, satisfies 
    \[
    |D\bar u_2(s)| \lesssim |s|^{-2+\delta},
    \]
    for an arbitrarily small $\delta > 0$.
    \end{theorem}
    For proof, see Section~\ref{sec:proofs-atom-model}. For now we merely state the formulae for $\hat u_1$ and $\hat u_2$ and defer the discussion on the derivation to the proof. In polar coordinates $x = r(\cos\theta,\sin\theta)$, the predictor $\hat u_1$ is given by 
    \begin{equation}\label{eqn-hatu1-formula}
        \hat u_1(x) := -\frac{K^3}{64} \phi^{(iv)}(0) r^{-1/2}\Big( \log r \sin \tfrac{\theta}{2} + \frac{1}{6} \sin\tfrac{5\theta}{2}\Big)
    \end{equation} 
    \begin{equation}\label{eqn-hatu2-simple-formula}
    \hat u_2(x) = C\frac{ \omega_2(x)}{\lvert x \rvert} = C r^{-1/2} \sin \tfrac{\theta}{2},
    \end{equation}
    where $C$ is a specific constant obtained from a summation over the whole lattice of terms related to a higher-order expansion of the lattice Green's function, which we will discuss in Section~\ref{sec:higher-order-G}. The precise formula for $\hat u_2$ is given in \eqref{def-hatu_2_proofs}.
    \begin{remark}\label{rem-incomplete-sinclair}
        A flexible boundary approach to atomistic modelling of fracture, as introduced by Sinclair in \cite{sinclair1975influence} and also discussed in \cite{sinclair1978flexible} (Section II.C, the \emph{Flex-S} approach) concerns dividing the computational domain into a core atomistic region, where atoms are free to vary and a transition and a far-field region, in which atoms are prescribed to follow a displacement determined by a truncated series of higher-order continuum expansion. In our case of an anti-plane model with a pair potential (thus ensuring isotropy), this expansion, as presented in \cite{stroh1958dislocations}, is given by, for $x = r(\cos\theta,\sin\theta)$ and the corresponding complex number $z = r\cos\theta + i r\sin\theta$,
        \[
        \sum_{j=0}^N c_j {\rm Imag}(z^{(\frac12-j)}) = \sum_{i=0}^N c_j r^{(\frac12-j)}\sin\left(\frac{(1-2j)\theta}{2}\right),
        \]
        The incompatibility of forces in the transition region is resolved by finding optimal prefactors $c_i$ which minimise the generalised forces, which are the variations with respect to the coefficients of atoms in the transition region. For this approach to be accurate, it is crucial that the expansion correctly captures the asymptotic behaviour of the equilibrium displacement.  In particular in the Sinclair's ansatz, the second term in the expansion, with behaviour of order $\mathcal{O}(r^{-1/2})$, is given by $c_1 r^{-1/2}\sin(-\theta/2)$ and the third, capturing $\mathcal{O}(r^{-3/2})$ behaviour, by $c_2 r^{-3/2}\sin(-3\theta/2)$. Our result shows that, up to log terms, the $\mathcal{O}(r^{-1/2})$ behaviour includes two extra terms that are missing from the Sinclair ansatz, namely 
        \[
        c_{1,2} \left(r^{-1/2}\log r \sin \tfrac{\theta}{2} + r^{-1/2} \tfrac{1}{6} \sin\tfrac{5\theta}{2}\right).
        \]
        They arise from the nonlinearity of the interaction potentials.
        We have thus established that the Sinclair-type approach to flexible boundary conditions is incomplete. To fix it, one should also optimise the prefactor $c_{1,2}$.
    \end{remark}
    \subsection{Higher order expansion of the Green's function}\label{sec:higher-order-G}
    To prove Theorem~\ref{thm:ubar2_decay}, a higher-order development for the associated lattice Green's function is also needed. We recall from \cite{2018-antiplanecrack} that the lattice Green's function
    \[
    G\,\colon\, \La \times \La \to \R
    \]
    corresponding to the model described in Section~\ref{sec:model} is defined as the solution to the pointwise discrete PDE
    \begin{equation}\label{eqn-G-pde}
        -{\rm Div}DG(m,s) = \delta(m,s), \quad G(m,s) = G(s,m),
    \end{equation}
    where the discrete divergence, which, for $g \, \colon\, \La \to \R^{\Rc}$, is defined as
    \begin{equation}\label{eqn-def-divergence}
        -{\rm Div}\,g(m) := \sum_{\rho \in \Rc} g_{\rho}(m-\rho) - g_{\rho}(m),
    \end{equation}
    and the gradient are both applied with respect to $m$. Note that if $g = Du$ for some $u\, \colon\,\La \to \R$, then the discrete divergence operator naturally respects the structure of $D$ defined in \eqref{eqn-def-Du}, in the sense that, for $u,v$ compactly supported, summation by parts holds, that is
    \[
    \sum_{m \in \La} (-{\rm Div}Du(m))v(m) = \sum_{m \in \La} Du(m) \cdot Dv(m) = \sum_{m\in \La}\sum_{\rho \in \Rc(m)}D_{\rho}u(m)D_{\rho}v(m).
    \]
    
    The following is already known.
    \begin{theorem}[adapted from \cite{2018-antiplanecrack}]\label{thm-G-old}
    There exists $G$ satisfying \eqref{eqn-G-pde}, such that, for any $\delta >0$, 
    \begin{equation}\label{eqn-d1d2G-old-decay}
    |D_1 D_2 G(m,s)| \lesssim (1+|\omega(m)||\omega(s)||\omega(m)-\omega(s)|^{2-\delta})^{-1},
    \end{equation}
    where $\omega$ is the complex square root mapping defined in \eqref{eqn-omega-def}. In particular, it admits a decomposition 
    \begin{equation}\label{eqn-G-old-decomp}
    G = \hat{G}_0 + \bar G_0,
    \end{equation}
    where 
    \begin{equation}\label{eqn-hatG0-formula}
    \hat{G}_0(m,s) = F(-\omega(m) + \omega(s)) + F(-\omega(m) + 
    \omega^*(s)), \quad F(x) = -\frac{1}{4\pi} \log|x|
    \end{equation}
    and $\bar G_0(\cdot,s) \in \Hcc$ (and hence, due to variable symmetry, in the other variable too). 
    \end{theorem}
    We note that $\hat G_0$ satisfies the decay rate in \eqref{eqn-d1d2G-old-decay} with $\delta = 0$. Establishing the same rate of decay, up to an arbitrarily small $\delta > 0$, for $\bar G_0$ is the main technical achievement of \cite{2018-antiplanecrack} and follows from a rather involved argument bearing resemblance to arguments in the regularity theory for elliptic PDEs \cite{beck2016elliptic}.
    
    In the present work, we rely on this zero-order analysis and refine it by realising that the far-field predictor $\hat G_0$ is not enough to obtain a corrector which decays faster everywhere away from the source point $s$. In fact, as part of our analysis, we will show that the decay rate in \eqref{eqn-d1d2G-old-decay}, with $\delta=0$, is sharp for $\bar G_0$ when $m \in B_{\frac{|s|}{8}}(0)$.
    
    To remedy this, we introduce the novel notion of a \emph{discrete geometry predictor}, $\hat G_1$, which accounts for the fact that near the crack tip, the complex square root mapping $\omega$, present in the definition of $\hat G_0$, distorts the lattice too much for $\hat G_0$ to correctly capture the discrete effects there. In particular, we decompose the full lattice Green's function $G$ from Theorem~\ref{thm-G-old} as
    \begin{align*}
    G &= \hat G_0 + \bar G_0\\
    &=\hat G_0 + \hat G_1 + \bar G_1,
    \end{align*}
    noting that in other words we have $\bar G_0 = \hat G_1 + \bar G_1$.
    
    \begin{definition}[Discrete geometry corrector $\hat G_1$]\label{def-hatG1}
    A function $\hat G_1\,\colon\,\La \times \La \to \R$ is called the \emph{discrete geometry corrector} if, for all $m,s \in \La$,
    \[
    -{\rm Div}D\hat G_1(m,s) =  \frac{-2\omega_2(s)}{\abs{s}}{\rm Div}D\omega_2(m),
    \]
    where, again, we recall that  $\omega_2$ is the second component of the complex square root mapping \eqref{eqn-omega-def}, and furthermore, for any $s \in \La$, it holds that $\hat G_1(\cdot,s) \in \Hcc$.
    \end{definition}
    
    To motivate this definition, we recall the explicit formula for $\hat G_0$ given in \eqref{eqn-hatG0-formula} and Taylor expand both terms in $F$ around $\omega(s)$ to obtain 
    \[
    -{\rm Div}D[\hat G_0 + \hat G_1](m,s) = -{\rm Div} D\hat G_1(m,s) -{\rm Div} D\omega(m) \cdot \left(\nabla F(-\omega(s)) + F(-\omega^*(s))\right) + \mathcal{O}(|s|^{-1}).
    \]
    Noting that 
    \[
    \nabla F(-\omega(s)) + F(-\omega^*(s)) = \left[0, 2\nabla_2 F(-\omega(s))\right]^{\rm T}\quad \text{ and }\quad 2\left(\nabla F(-\omega(s))\right)_2 = \frac{-2\omega_2(s)}{\abs{s}}
    \]
    and that we would like the lowest order term to cancel out with $-{\rm Div}D\hat G_1(m,s)$, we arrive at the equation in Definition~\ref{def-hatG1}. Note that the problem of finding $\hat G_1$ is non-trivial since $\omega_2 \not\in \Hcc$.
    
    We prove the following.
    \begin{prop}\label{prop-ghat1}
        There exists a discrete geometry predictor $\hat G_1\,\colon\,\La \times \La \to \R$ in the sense of Definition~\ref{def-hatG1} and it admits a decomposition
        \begin{equation}\label{eqn-hat-G1_decomp}
            \hat G_1(m,s) = \hat G^m_1(m) \hat G^s_1(s),
        \end{equation}
        where $\hat G^s_1(s) = \frac{-2\omega_2(s)}{\abs{s}}$ and $G^m_1 \in \Hcc$ satisfies
        \[
        |D^i \hat G^m_1(\ell)| \lesssim |\ell|^{-1/2-i + \delta}
        \]
        for $i=0,1$.
    \end{prop}
    The proof will be presented in Section~\ref{sec:proof-prop-ghat1}.
    
    As the name implies, the discrete geometry corrector is specifically important when $\lvert m \rvert$ is small and thus close to the crack tip. To account for that we additional introduce a cutoff function $\mu$ (see \eqref{eqn-def-mu} for a precise definition) and adjust $\hat G_1$ to
    
    \begin{equation}
        \hat G_{1, \mu}(m,s) = \mu(m,s) \hat G^m_1(m) \hat G^s_1(s).
    \end{equation}
    In particular, we then have the remainder $\bar G_{1, \mu}(m,s)$ defined by
    \begin{equation}
        G = \hat G_0 + \hat G_{1, \mu}+ \bar G_{1, \mu}.
    \end{equation}
    This does indeed allow us to resolve the terms on order $O(\lvert s \rvert^{-3/2})$ as we state in the following theorem.
    \begin{theorem}\label{thm-Gbar1-decay}
        It holds that
        \begin{equation}  \label{eq-mixedGbar-smallm}
            \lvert D_1 D_2 \bar{G}_{1, \mu}(\ell,s) \rvert \lesssim \lvert s\rvert^{-2+\delta} \lvert \ell \rvert^{-1/2}
        \end{equation}
    for $\lvert \ell \rvert \leq \lvert s \rvert/16$.
    \end{theorem}
    See Section \ref{sec:proofsgreensfunction} for the proof.

    
    \section{Numerical approximation}\label{sec:numerics}
    In this section we discuss numerical approximations to the infinite lattice problem from Theorem~\ref{thm:ubar0} and present how the results of Section~\ref{sec:higher_pred_u} can be used to supply more accurate boundary conditions for numerical simulations. The setup largely mimics the one presented in \cite{2017-bcscrew,2018-antiplanecrack} and is as follows.
    
    We introduce a generalised energy difference functional
    \[
        \E(u_{\rm pred}^{(i)}, u) := \sum_{m \in \La}\sum_{\rho \in \Rc(m)} \phi(D_{\rho}u_{\rm pred}^{(i)}(m) + u(m)) - \phi(D_{\rho} \hat u_0(m)),
    \]
    where $u_{\rm pred}^{(i)} := \sum_{k=0}^i \hat u_k$ and we will directly compare the cases when $i=0,1,2$. It is immediate that finding
    \begin{equation}\label{eqn-baru_i-problem}
    \bar u_i \in {\rm arg min}\{ \E(u_{\rm pred}^{(i)}, u) \mid u \in \Hcc\}
    \end{equation}
    is equivalent to finding ${\rm arg min}_{u \in \Hcc} \E(u)$ defined in \eqref{eqn-E-def} via the the expansion \eqref{eqn-u-expansion}.
    
    The object of our study in this section is a supercell approximation to \eqref{eqn-baru_i-problem} constrained to a finite domain $\La_R := B_R \cap \La$ and with $u_{\rm pred}^{(i)}$ prescribed as a boundary condition on $\La \setminus \La_R$. It can be written as a Galerkin approximation given by 
    \begin{equation}\label{eqn-baru_iR-problem}
    \bar u_i^R \in {\rm argmin}\{ \E(u_{\rm pred}^{(i)}, u) \mid u \in \mathcal{H}^0(\La_R)\},
    \end{equation}
    where
    \[
    \mathcal{H}^0(\La_R) := \{ v\, \colon\, \La \to \R \mid v = 0\; \text{ in } \La \setminus \La_R\}.
    \]
    The following result is an immediate consequence of Theorem~\ref{thm:ubar2_decay}, with the proof as in \cite[Theorem~3.1]{2017-bcscrew}, which in itself in a concise restatement of a corresponding result in \cite{EOS2016}.
    \begin{theorem}
        Let $\bar u_2$ be a strongly stable solution to \eqref{eqn-baru_i-problem}, meaning that there exists some $C>0$ such that 
        \[
            \delta_u^2 \E(u_{\rm pred}^{(i)}, u)[v,v] \geq C\|v\|_{\Hcc}^2,
        \]
        for all $v$ with compact support. Then there exist $C, R_0 > 0$, such that for all $R > R_0$ there exists a stable solution $\bar u_i^R$ to \eqref{eqn-baru_iR-problem} satisfying 
        \[
        \|\bar u_2^R - \bar u_2\|_{\Hcc} \lesssim R^{-1}.
        \]
    \end{theorem}
    This is a significant improvement upon the known convergence rate $\|\bar u_0^R - \bar u_0\|_{\Hcc} \lesssim R^{-1/2 + \beta}$, where $\beta > 0$ is arbitrarily small, as established in \cite[Theorem~2.8]{2018-antiplanecrack}.
    
    \subsection{Numerical tests}
    Based on the presented numerical approximation framework, we now discuss numerical tests we have conducted, through which we numerically verify the result of Theorem~\ref{thm:ubar2_decay}. We also do a slightly broader comparison in which the boundary conditions are set to be given by (i) $\hat u_0$, (ii) $\hat u_0 + \hat u_1$ and (iii) $\hat u_0 + \hat u_1 + \hat u_2$.
    
    In particular, we consider a finite domain $\La \cap B_R$ with $R=256$, resulting in 200772 atoms being simulated and focus on the case where the pair potential is given by
    \[
    \phi(r) = \frac16\left(1-\exp(-3r^2)\right),
    \]
    as considered in \cite{2018-antiplanecrack,2019-antiplanecrack,BK2021}. We further prescribe the stress intensity factor $K$, which enters as a prefactor in $\hat u_0$ (c.f. \eqref{eqn-def-uhat0}) and in $\hat u_1$ (c.f. \eqref{eqn-hatu1-formula}), to be $K=0.4$, which is below the critical $K_G \approx 0.49$ reported in \cite[Table 1]{BK2021}, thus ensuring the crack remains at the centre of the domain. 
    
    To prescribe the full first order predictor $u_{\rm pred}^{(2)} = \hat u_0 + \hat u_1 + \hat u_2$, we need to compute the constant entering as a prefactor in the formula for $\hat u_2$ in \eqref{eqn-hatu2-simple-formula}. We obtain it with a bisection approach, which is sufficient for our purposes, but it is certainly the aim to make this procedure automated, e.g. by combining it with the NCFlex scheme from \cite{BK2021} and treating this constant as an extended variable of the system. 
    
    To find correctors $\bar u_i$ we, we employ the Julia library \textsc{Optim.jl} \cite{mogensen2018optim} with their implementation of the Conjugate-Gradient algorithm with line searches, terminating at $\ell_{\infty}$-residual of $10^{-8}$.
    
    To elucidate the intermediate impact of introducing the predictor $\hat u_1$, not only do we present the decay of the corrector $|D\bar u_i|$, but also of the forces when atoms are displaced according to $u = u_{\rm pred}^{(i)}$, which is given by  $|{\rm Div}\nabla V(Du_{\rm pred}^{(i)})|$, where we use the concise notation
    \[
    V(Du(m)) = \sum_{\rho \in \Rc} \phi((Du(m))_{\rho}).
    \]
    Finally, we also plot the decay of the linear residual $|{\rm Div} \nabla^2V(0)D\bar u_i(m)|$, where we note that since $\phi''(0) = 1$, we have $\nabla^2V(0) = {\rm Id}$. The results are presented in Figure~\ref{fig1}.
    
    \begin{figure}[htbp]
        \begin{subfigure}[c]{0.31\textwidth}
            \centering
                    {{\includegraphics[trim = 0 0 0 0,clip,width=\textwidth]{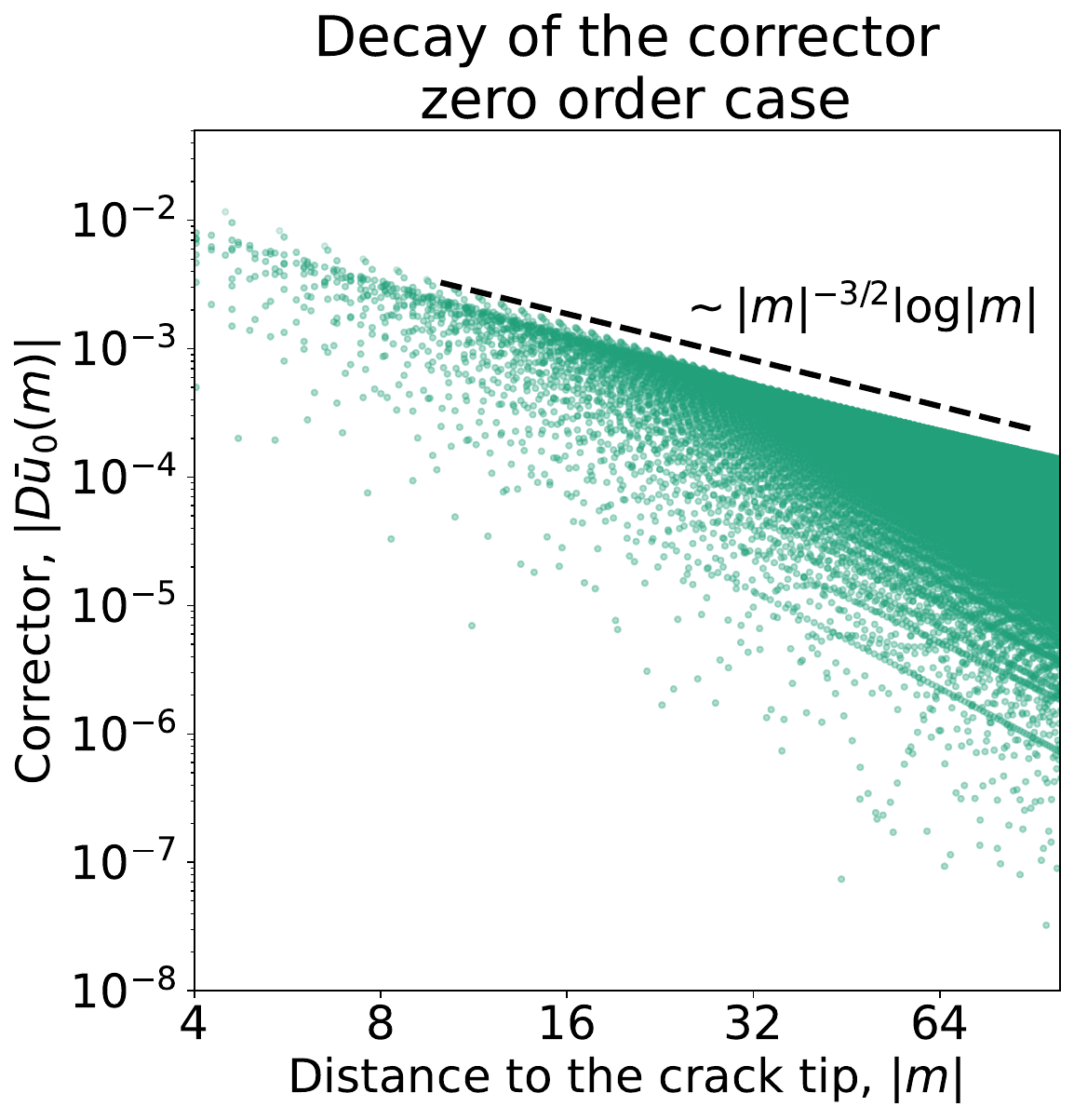}}}
        \end{subfigure}
        \hfill
        \begin{subfigure}[c]{0.31\textwidth}
            \centering
                    {{\includegraphics[trim = 0 0 0 0,clip,width=\textwidth]{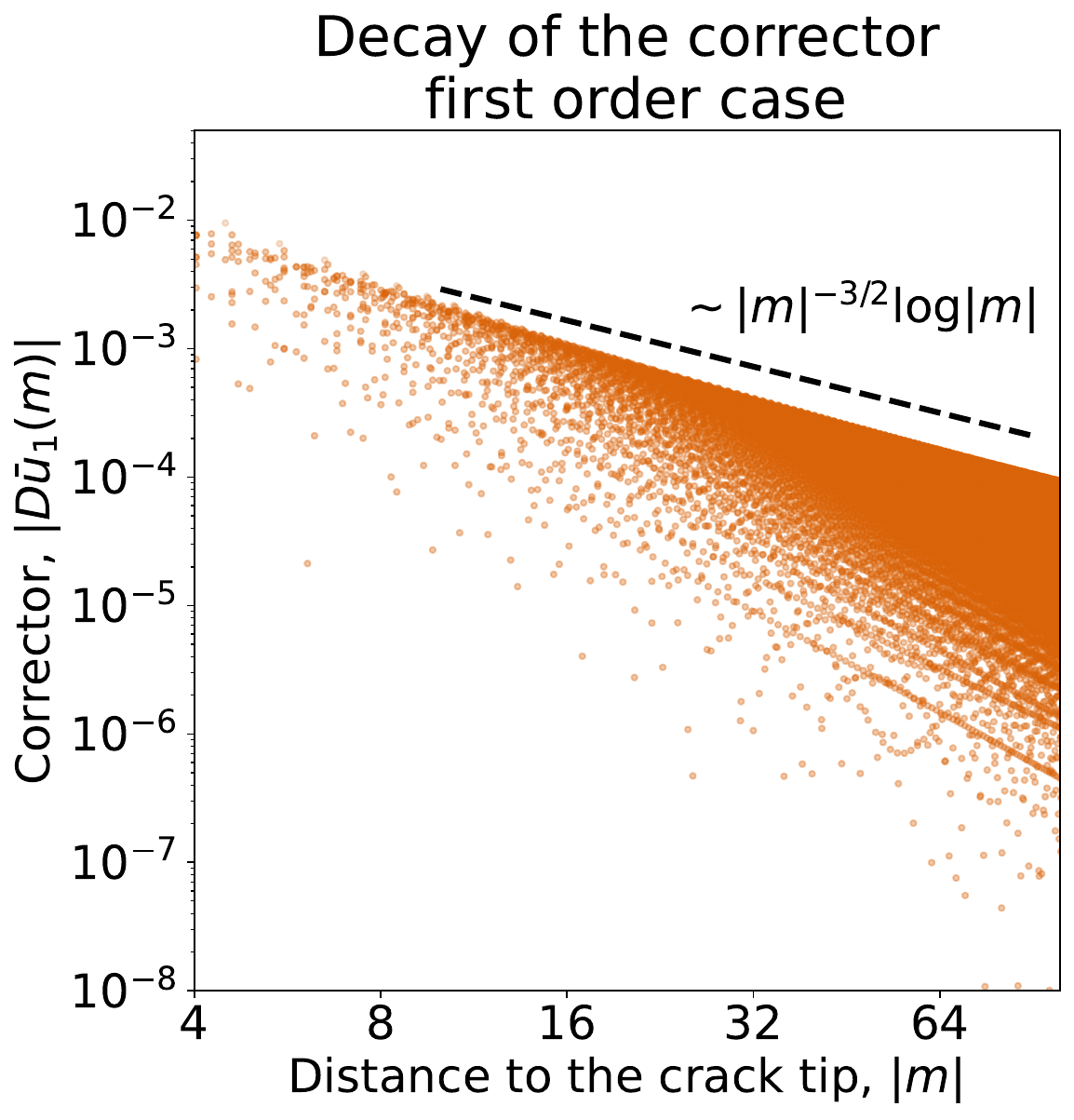}}}
        \end{subfigure}
        \hfill
        \begin{subfigure}[c]{0.31\textwidth}
            \centering
                    {{\includegraphics[trim = 0 0 0 0,clip,width=\textwidth]{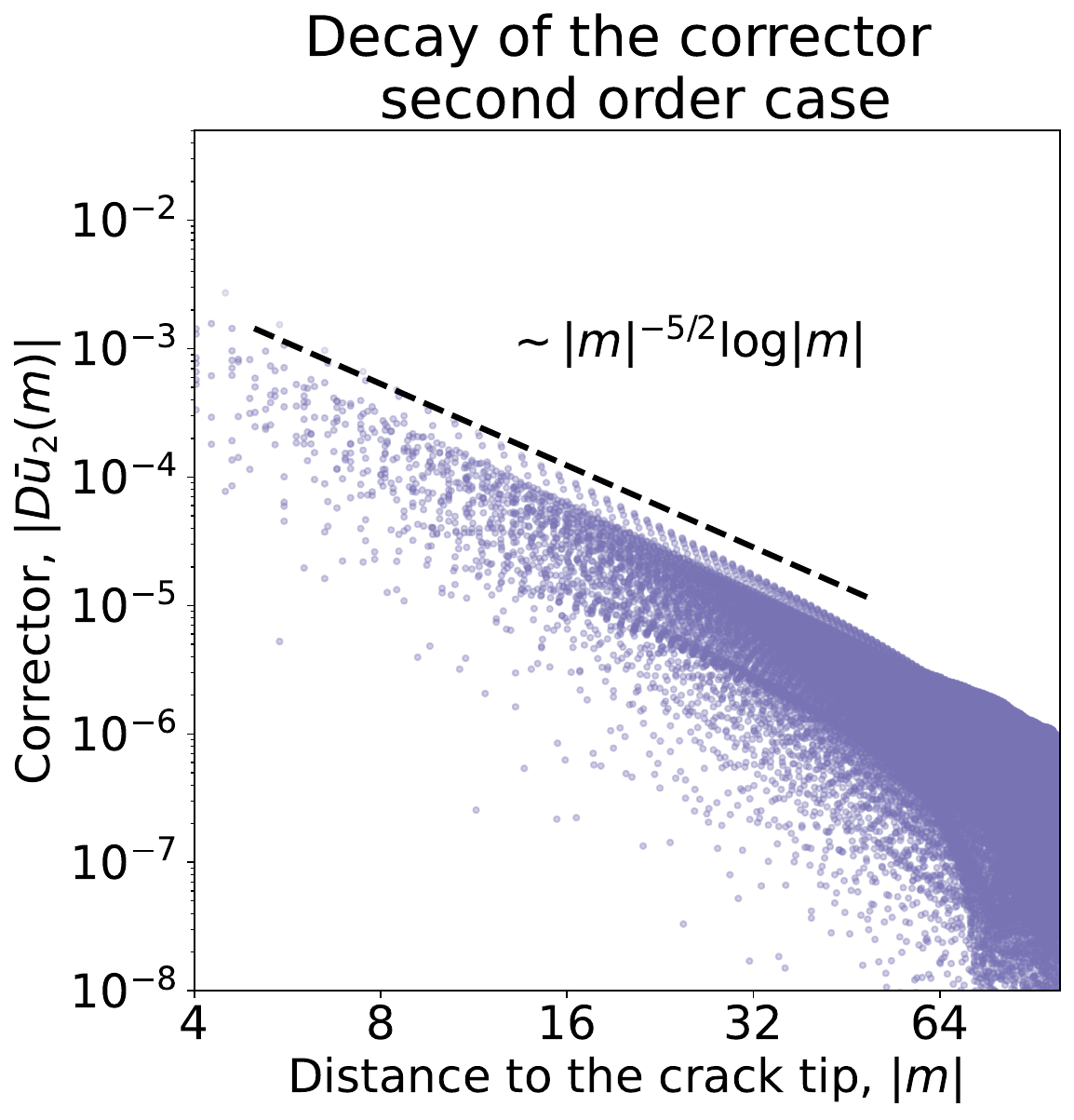}}}
        \end{subfigure}
        \hfill\\
        \centering
        \begin{subfigure}[c]{0.31\textwidth}
            \centering
                    {{\includegraphics[trim = 0 0 0 0,clip,width=\textwidth]{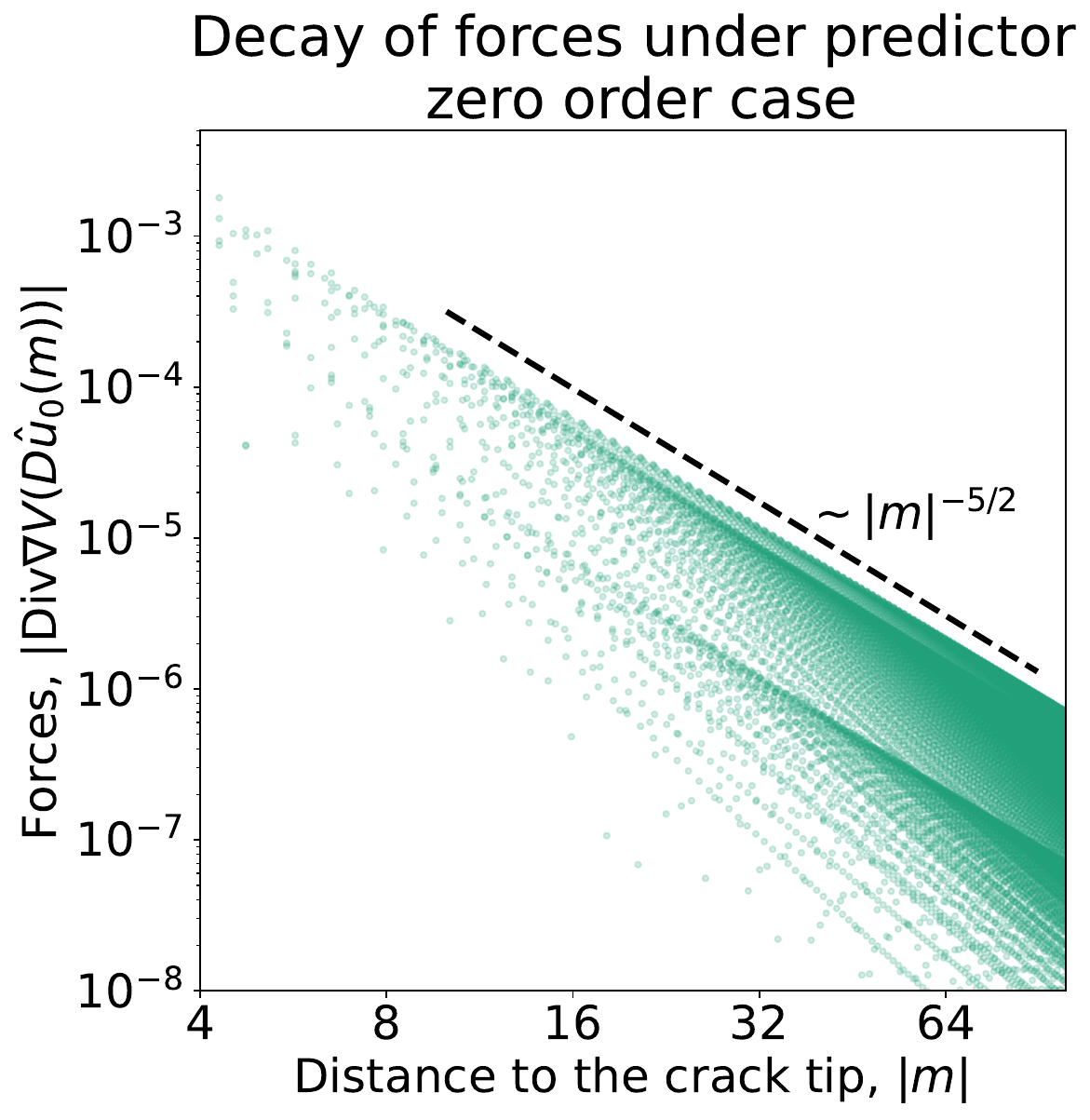}}}
        \end{subfigure}
        \hfill
        \begin{subfigure}[c]{0.31\textwidth}
            \centering
                    {{\includegraphics[trim = 0 0 0 0,clip,width=\textwidth]{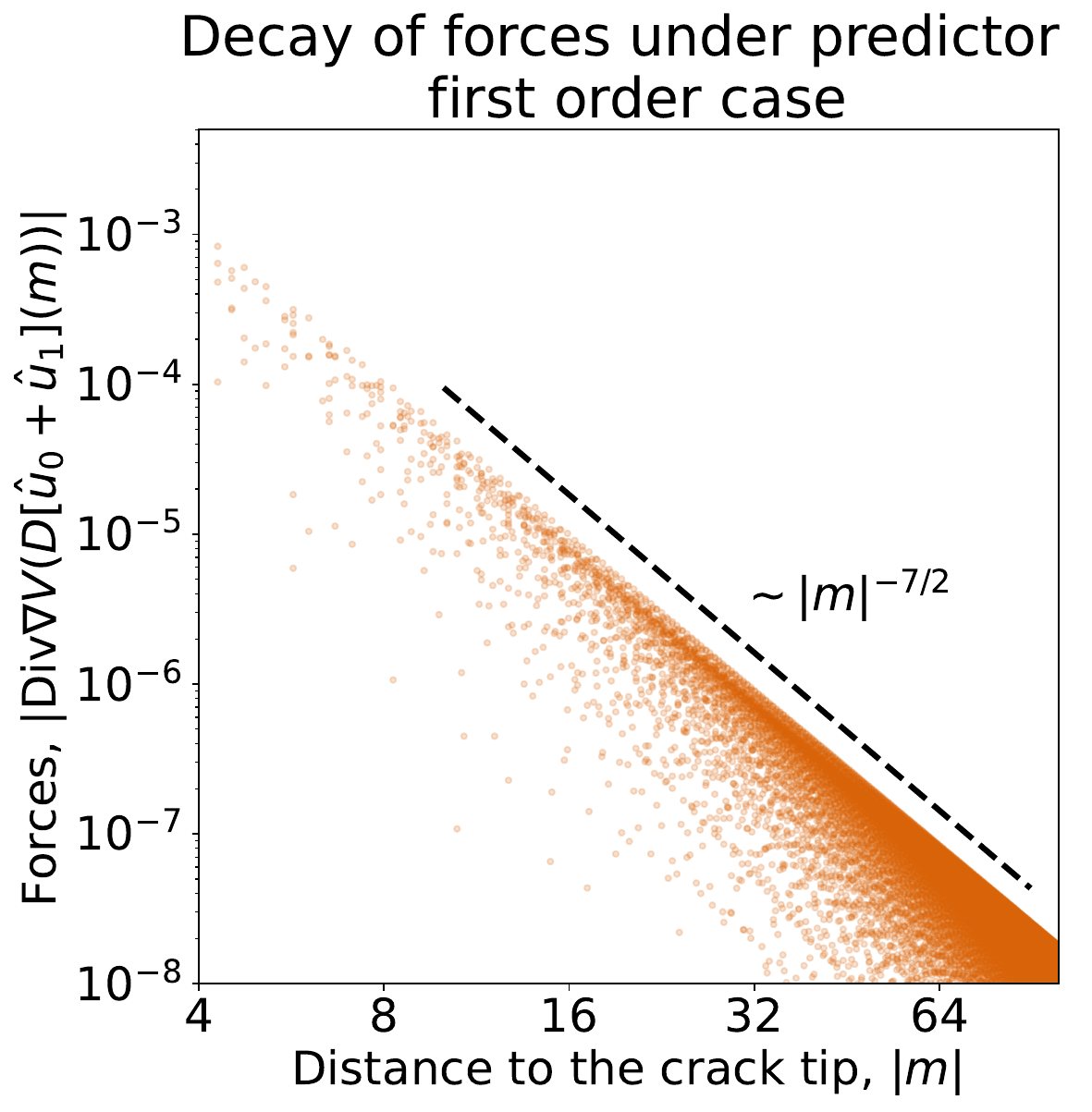}}}
        \end{subfigure}
        \hfill
        \begin{subfigure}[c]{0.31\textwidth}
            \centering
                    {{\includegraphics[trim = 0 0 0 0,clip,width=\textwidth]{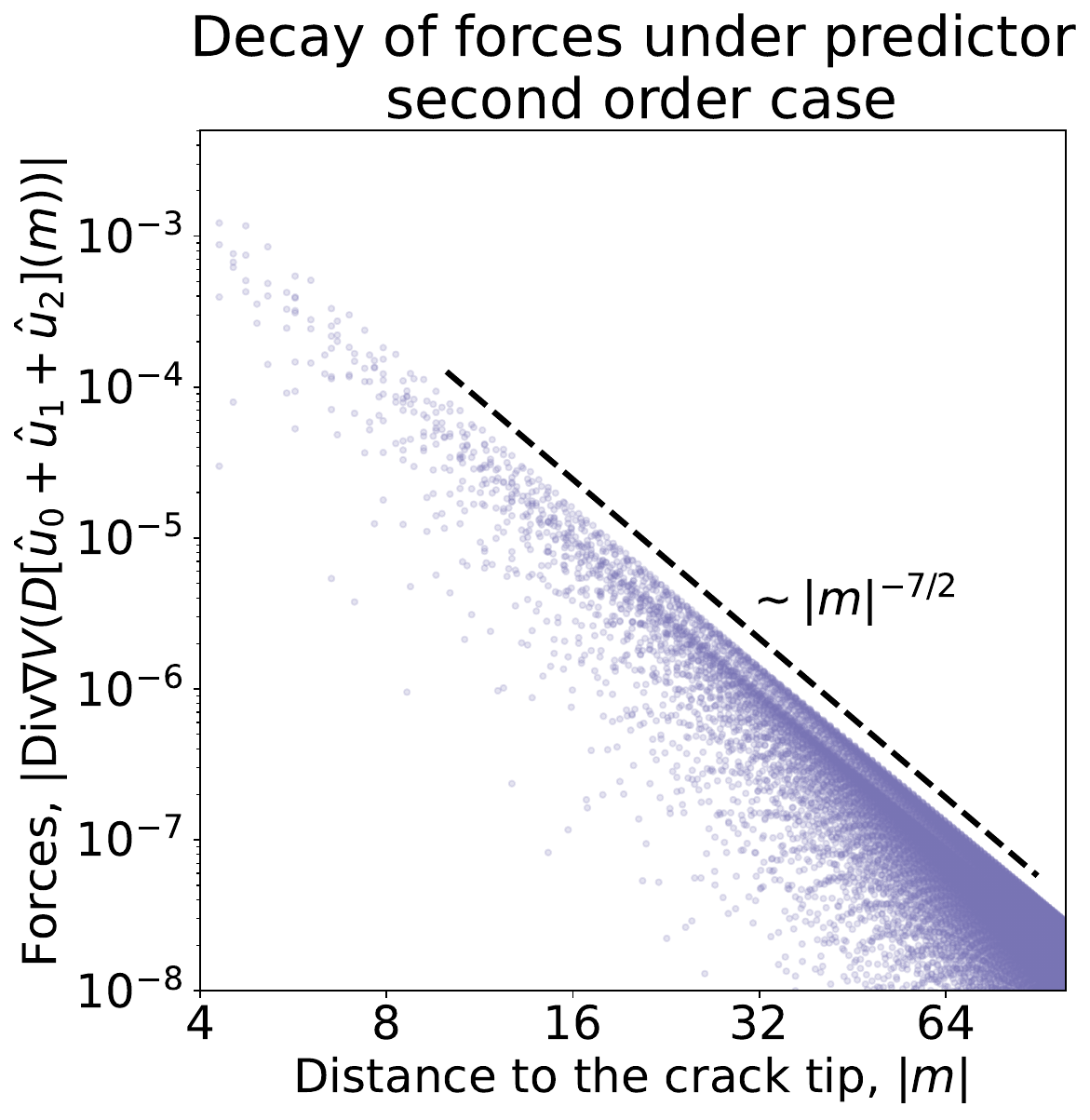}}}
        \end{subfigure}
        \hfill\\
        \centering
        \begin{subfigure}[c]{0.31\textwidth}
            \centering
                    {{\includegraphics[trim = 0 0 0 0,clip,width=\textwidth]{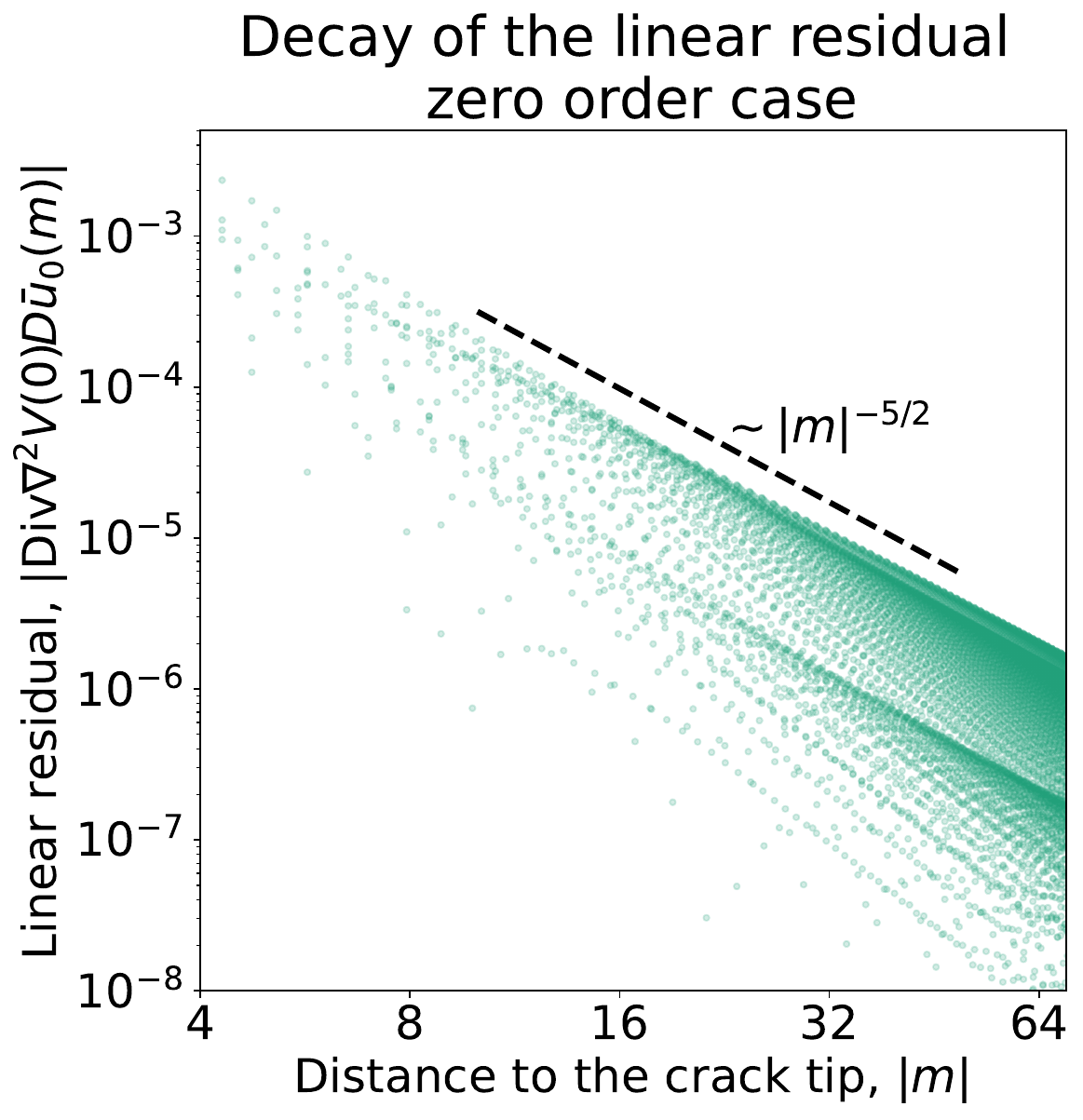}}}
        \end{subfigure}
        \hfill
        \begin{subfigure}[c]{0.31\textwidth}
            \centering
                    {{\includegraphics[trim = 0 0 0 0,clip,width=\textwidth]{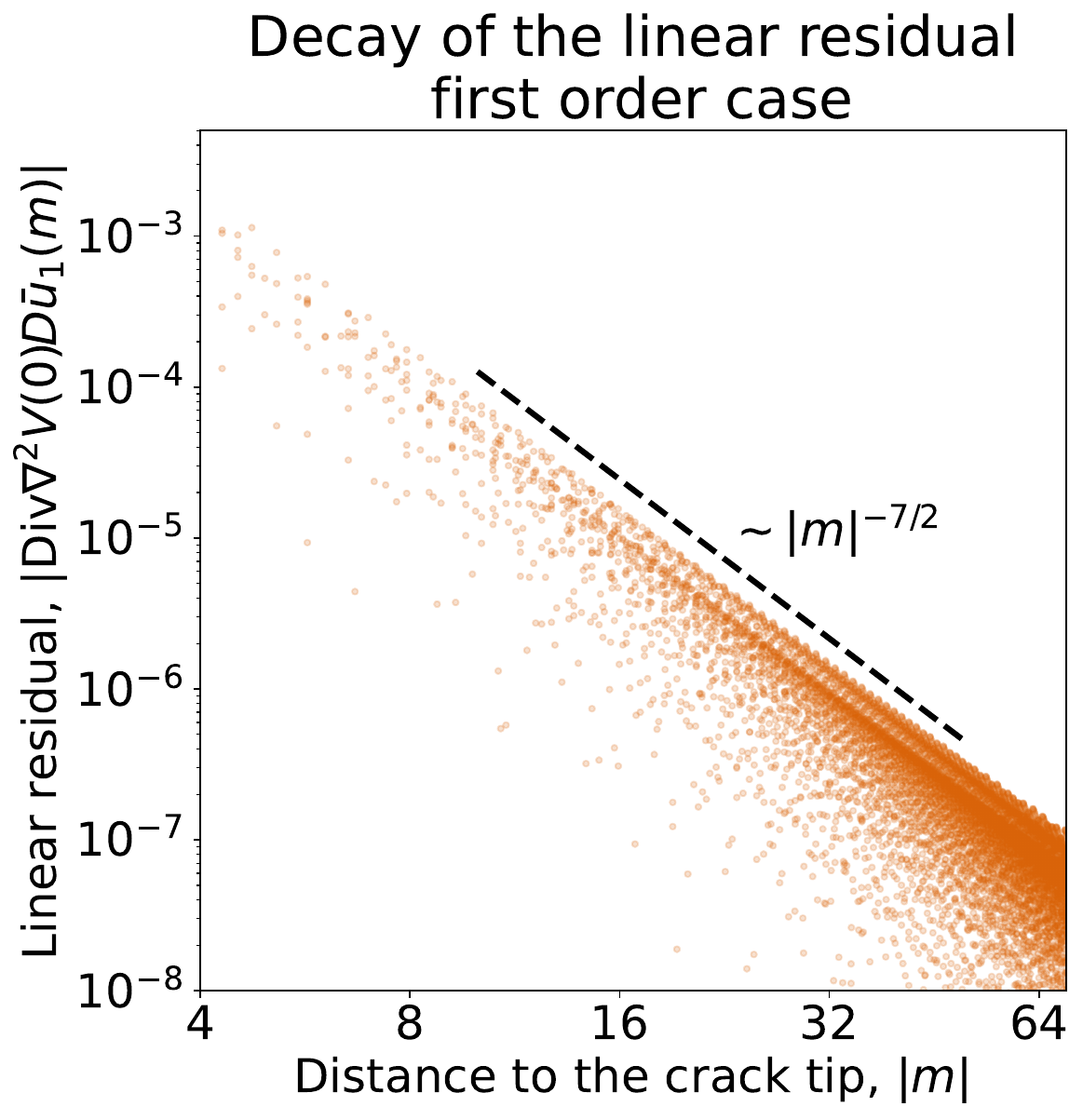}}}
        \end{subfigure}
        \hfill
        \begin{subfigure}[c]{0.31\textwidth}
        \centering
                    {{\includegraphics[trim = 0 0 0 0,clip,width=\textwidth]{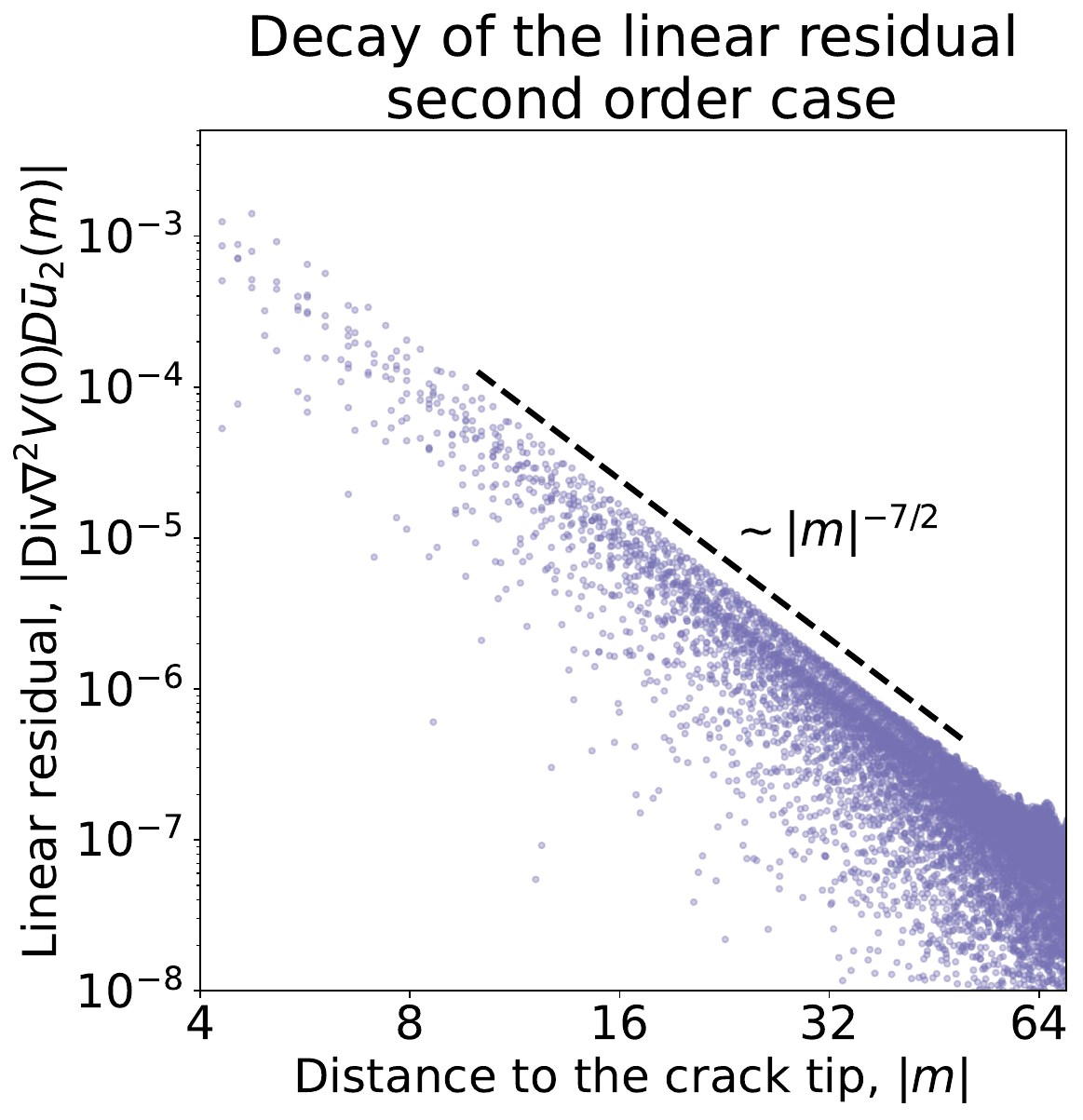}}}
        \end{subfigure}
        \hfill
        \caption{Top row: the decay of $|D\bar u_i(m)|$ for $i=0,1,2$ (left to right). Middle row: the decay of forces $|{\rm Div}\nabla V(Du_{\rm pred}^{(i)})|$ for $i=0,1,2$ (left to right). Bottom row: the decay of the linear residual $|{\rm Div} \nabla^2V(0)D \bar u_i(m)|$ for $i=0,1,2$ (left to right). We can clearly see that while prescribing $u_{\rm pred}^{(1)} = \bar u_0 + \bar u_1$ as a boundary condition improves the decay of forces (middle row) and the linear residual (bottom row), it fails to improve the decay of $|D \bar u_1(m)|$ and additionally $\hat u_2$ is needed.} 
           \label{fig1}
    \end{figure}
    
    We note that while Theorem~\ref{thm:ubar2_decay} suggests that, up to log terms, we have $|D\bar u_2(m)| \lesssim |m|^{-2}$, we in fact observe an improved rate
    \[
    |D\bar u_2(m)| \lesssim |m|^{-5/2}.
    \]
    We are at present unsure whether this is generic behaviour to be expected (akin to how considering mirror symmetry in \cite{2017-bcscrew} led to proving an improved decay, which was first numerically observed in \cite{HO2014}), or whether it is a numerical artefact of our finite-domain approximation. We hope to address this issue in future work. In any case, the simulations confirm that $\hat u_0 + \hat u_1 + \hat u_2$ exactly captures the far-field $\sim |m|^{-3/2}$ behaviour, which was our primary aim. Note however, that to prove such an improved decay result, new non-trivial ideas seem to be needed.

    
    \section{Conclusions and outlook}\label{sec:conclusions}
    
    We have fully developed the first order correction term in the atomistic asymptotic expansion of the elastic field around a Mode III crack in anti-plane geometry. As a result the well-known flexible boundary condition ansatz due to Sinclair and co-authors has been proven to be incomplete, which can have far-reaching consequences for the atomistic simulations of fracture. To obtain our rigorous proofs, we have also succeeded in providing a higher order description of the associated lattice Green's function. In the process, we have thus shown a first explicit example of how to expand the theory developed in \cite{BHO22} to more complex geometries that can no longer rely on the homogeneous lattice Green's function but has to account for genuine spatial inhomogeneity.
    
    In the future, we aim to use the ideas developed here to heuristically derive a correct atomistic asymptotic expansion for a variety of fracture models, including for Mode I vectorial models and combine this ansatz with numerical continuation tools developed in \cite{BK2021}.
    \section{Proofs: Setup, prerequisites, and tools}\label{sec:proofs_basics}
    \subsection{Cut-offs}
    \begin{definition}[Cut-off functions]\label{def-cutoff-fcts}
        Consider a generic scalar cut-off function $\hat \eta \,\colon\, [0,\infty] \to \R$ satisfying $\hat \eta(x) = 1$ on $\left[0,\tfrac12\right]$ and $\hat \eta(x) = 0$ for $x \in [1, \infty)$ and smooth and decreasing in-between. The class of lattice cut-off function we will employ are $\eta\,\colon\, \La \to \R$ defined as 
        \begin{equation}\label{eqn-lattice-cutoff-fct}
            \eta(m):= \hat\eta\left(\tfrac{|m-l|}{c|s|}\right),
        \end{equation}
        where the constant $c$ and two lattice points $\ell, s \in \La$ will always be specified. As a result, $D\eta$ is only non-zero in an annulus that scales like $|s|$ and in particular $|D^j\eta(x)| \lesssim |s|^{-j}$.
    \end{definition}
    \begin{lemma}\label{lem:pushing_eta}
        Suppose $u,v,\eta\,\colon \La \to \R$ and $\eta$ is compactly supported. Then
        \begin{align*}
            \sum_{m \in \La} Du(m) \cdot D[v \eta](m) &= \sum_{m \in \La} D[u\eta](m) \cdot Dv(m)\\
            &+ \sum_{m \in \La} v(m)\left(D\eta (m) \cdot D u(m)\right) - u(m)\left(D\eta \cdot D v(m)\right).
        \end{align*}
    \end{lemma}
    \begin{proof}
       The result follows directly from realising that we can rewrite $D_{\rho}[v\eta](m)$ as
       \begin{align*}
        D_{\rho}[v\eta](m) &= u(m+\rho)\eta(m+\rho) - v(m)\eta(m)\\  &= v(m+\rho)\eta(m+\rho) \underbrace{- v(m)\eta(m+\rho) + v(m)\eta(m+\rho)}_{=0} - v(m)\eta(m)\\ &= D_{\rho}v(m)\eta(m+\rho) + v(m)D_{\rho}\eta(m)
       \end{align*}
    \end{proof}
    and applying this rewrite to $D[v\eta]$ on the left-hand side and to $D[u\eta]$ on the right-hand side.
    \subsection{Properties of the function spaces}
    Let us collect a few important properties of functions in $\Hcc$. Remember that we set $(Du(m))_\rho =0$ when the bond crosses the crack. The following statements closely follow the full homogeneous case though, despite this difference.
    \begin{lemma}[Poincaré on sliced annulus]\label{lem:discretePoincare}
    There are $C_P, R_P>0$ such that for all $R>0$ large enough and all $v$
        \[
        \|v - (v)_{\Sigma_R}\|_{\ell^2(\Sigma_R)} \lesssim R\|Dv\|_{\ell^2(\Sigma_R')},
        \]
        where $(v)_{\Sigma_R} = \frac{1}{\lvert \Sigma_R \rvert}\sum_{m \in \Sigma_R}v(m)$ and $\Sigma_R := (B_{c_2R}(0) \setminus B_{c_1R}(0)) \cap \Lambda$, $\Sigma_R' := (B_{c_2R+R_P}(0) \setminus B_{c_1R-R_P}(0)) \cap \Lambda$.
    \end{lemma}
    \begin{proof}
        This is basically \cite[Lemma 7.1]{EOS2016pp}. Note however that we use the sliced annulus which in our notation means $Du$ is set to $0$ for bonds that cross the crack. This does not impact the proof in any meaningful way though, as the continuum result on the sliced annulus holds true.
    \end{proof}
    \begin{lemma}[Density of compactly supported functions.] \label{lem:density}
        $\Hc$ is dense in $\Hcc$ with respect to the $\lVert Dv \lVert_{\ell^2}$ semi-norm.
    \end{lemma}
    \begin{proof}
        This is a standard consequence of Lemma \ref{lem:discretePoincare} when combined with a cutoff as in Definition \ref{def-cutoff-fcts}. Indeed on just sets $v_n(m) = (v(m)-(v)_{\Sigma_n}) \eta(m)$ with $\eta(m) = \hat\eta(\frac{\lvert m \rvert}{n})$ and the annulus based on $c_1=1/4$, $c_2=5/4$. Then one can calculate
        \begin{align*}
            \lVert Dv - Dv_n \rVert_{\ell^2} &\lesssim \lVert ((Dv)_\rho (1-\eta(m+\rho)))_{\rho \in \Rc} \rVert_{\ell^2} + \lVert (v(m)-(v)_{\Sigma_n}) D\eta(m) \rVert_{\ell^2}\\
            &\lesssim \lVert Dv \rVert_{\ell^2(\Lambda \setminus B_{5n/4}(0))} + \frac{1}{n} \lVert v - (v)_{\Sigma_n}\rVert_{\ell^2(\Sigma_n)}\\
            &\lesssim \lVert Dv \rVert_{\ell^2(\Lambda \setminus B_{5n/4}(0))} +  \lVert Dv \rVert_{\ell^2(\Sigma_n')}\\
            &\lesssim \lVert Dv \rVert_{\ell^2(\Lambda \setminus B_{n/8}(0))}.
        \end{align*}
        And this last term goes to zero as $n \to \infty$.
    \end{proof}
    \begin{lemma}[Discrete Sobolev Embedding]\label{lem:sobolev_emb}
    If $Dv \in \ell^p$ for some $1 \leq p<d=2$ then there exists a constant $v_{\infty}$ such that $v-v_{\infty} \in \ell^{p^\ast}$ where $p^\ast = \frac{2p}{2-p}$ is the Sobolev conjugate.
    \end{lemma}
    \begin{proof}
    Without the crack in the geometry, the statement can be found in \cite[Prop. 12(iii)]{OrtnerShapeev12}.
    Note though that now $Dv$ is set to $0$ when crossing the crack $\Gamma_0$.
    
    It holds that 
    \[\lVert f \rVert_{L^{p^\ast}(\R^2 \backslash \Gamma_0)} \lesssim \lVert f \rVert_{W^{1,p}(\R^2 \backslash \Gamma_0)}\]
    for all $f \in W^{1,p}(\R^2 \backslash \Gamma_0)$ according to \cite{AF77}, as $\R^2 \backslash \Gamma_0$ satisfies the cone condition.
    A simple rescaling $g_t(x)=f(tx)$ for $t>0$ then shows
    \[t^{1-\tfrac{2}{p}} \lVert f \rVert_{L^{p^\ast}(\R^2 \backslash \Gamma_0)} \lesssim   t^{-\tfrac{2}{p}} \lVert f \rVert_{L^p(\R^2 \backslash \Gamma_0)} + t^{1-\tfrac{2}{p}} \lVert \nabla f \rVert_{L^p(\R^2 \backslash \Gamma_0)}.\]
    
    Dividing by $t^{1-\tfrac{2}{p}}$ and sending $t \to \infty$ reduces this to
    \[\lVert f \rVert_{L^{p^\ast}(\R^2 \backslash \Gamma_0)} \lesssim \lVert \nabla f \rVert_{L^{p}(\R^2 \backslash \Gamma_0)},\]
    where we used that $\R^2 \backslash \Gamma_0$ is preserved under the rescaling.
    
    Now, for any discrete $v \in \Hc$ we can use an interpolation (e.g., the one discussed in Section \ref{sec:interpol}) to conclude that
    \[ \lVert v \rVert_{\ell^{p^\ast}(\Lambda)} \lesssim \lVert D v \rVert_{\ell^{p}(\Lambda)}.\]
    The statement then follows from the the density of $\Hc$, Lemma \ref{lem:density}.
    \end{proof}
    \begin{lemma}\label{lem:decay_with_sobolev}
        Suppose $v\,\colon\,\La \to \R$ satisfies
        \[\lvert Dv(m) \rvert \lesssim \lvert m \rvert^{-\alpha}\]
        for some $\alpha>1$. Then there exists a constant $v_{\infty}$, such that
        \[\lvert v(m)-v_{\infty} \rvert \lesssim \lvert m \rvert^{1-\alpha}.\]
    \end{lemma}
    \begin{proof}
    Choosing any $p$ with $\frac{2}{\alpha}<p<2$ we can apply Lemma \ref{lem:sobolev_emb} and in particular obtain a shift $v_{\infty}$ such that $v-v_{\infty} \to 0$ at infinity. A simple telescope sum to infinity that does not cross the crack then proves the result.
    \end{proof}
    
    \subsection{Interpolation of lattice functions}\label{sec:interpol}
    Given a lattice function $v\, \colon\,\La \to \R$, it will be useful to define an interpolation operator to obtain $Iv\, \colon\, \R^2\setminus \Gamma_0 \to \R$. We begin by recalling how this was handled in \cite[Section~4.2.1.]{2018-antiplanecrack}, followed by necessary modifications needed for our purposes. 
    
    We begin by defining 
    \begin{equation}\label{eqn-regions-of-domain}
    \Omega_\Gamma := \left\{x \in \R^2 \mid x_1 \leq \tfrac12, \; x_2 \in \left(-\tfrac12,\tfrac12\right)\right\} \setminus \Gamma_0,\quad \Omega_0 := \Omega_{\Gamma} \cap \left[-\tfrac12,\tfrac12\right]^2, \quad Q_0 := \Omega_0 \cap \Gamma. 
    \end{equation}
    The default interpolation operator $I$ is defined as follows. 
    On $\R^2 \setminus \Omega_\Gamma\cup \Gamma_0$, that is away from the discrete crack surface $\Gamma$, the squares in the lattice are carved into two right-angle triangles along one of the two diagonals (e.g. top left to bottom right) and $I$ is defined as the (P1) piecewise linear interpolation operator. 
        
    We also want $Iv$ to be well-defined on $\Omega_\Gamma$ and continuous across $\Gamma$. Away from the crack tip this is is achieved by extending it so that it aligns with the the values of $Iv(x)$ for $x \in \Gamma$ and is constant in the normal direction. At the crack tip, inside $Q_0$ (defined in \eqref{eqn-regions-of-domain}), we create two new interpolation points, one precisely at the crack tip $x_0 := (0,0)$ and the other at $\tilde x:= (-\tfrac12,0) \in \Gamma_0$. We define the interpolation there, $Iu(x_0)$, as the average of the four lattice points inside $\overline Q_0$ and further $\lim_{x \to \tilde x^\pm}Iv(x) = v(\tilde x_\pm)$, where $\tilde x_\pm = \tilde x + (0,\pm \tfrac12)$. By construction, the resulting P1 interpolant does not need to be continuous across $\Gamma_0$.
    
    \begin{figure}[htbp]
        \begin{subfigure}[c]{0.48\textwidth}
            \centering
                    {{\includegraphics[trim = 0 0 0 0,clip,width=\textwidth]{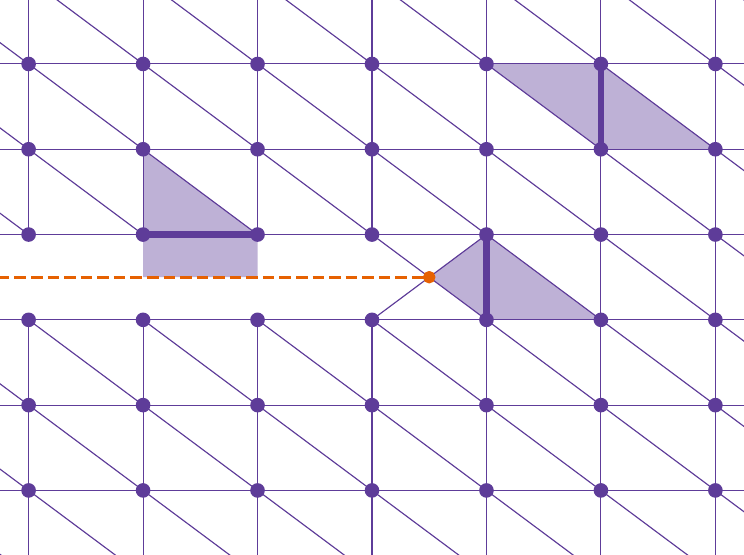}}}
        \end{subfigure}
        \hfill
        \begin{subfigure}[c]{0.48\textwidth}
            \centering
                    {{\includegraphics[trim = 0 0 0 0,clip,width=\textwidth]{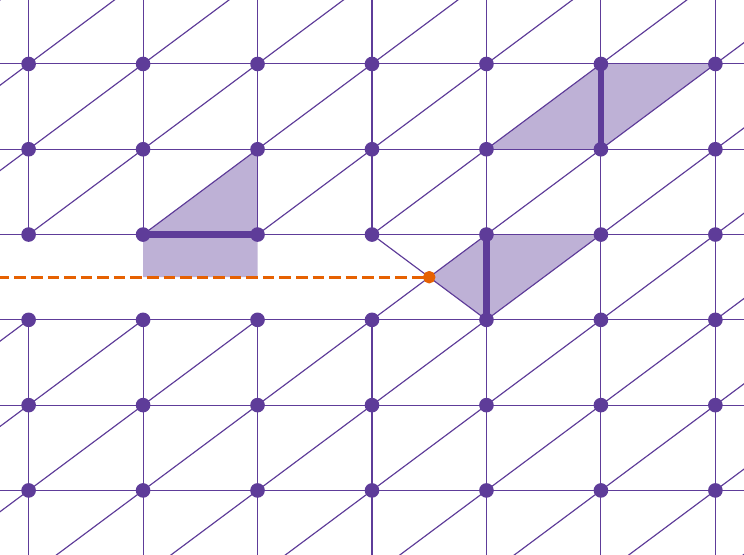}}}
        \end{subfigure}
    \caption{The default interpolation (left) and the mirror interpolation (right). Three example bonds $b(m,\rho)$ are highlighted together with regions $R_{m,\rho}$ associated with it.}
           \label{fig2}
    \end{figure}
    
    Under this interpolation, to each bond $b(m,\rho)$ for $m \in \La$, $\rho \in \Rc(m)$ we can associate the following regions. Let $U_{m,\rho}$ denote the union of all right-hand triangles for which the bond $b(m,\rho)$ is an edge; let $Q_{m,\rho}$ denote the union of all rectangular strips inside $\Omega_\Gamma$ for which the bond $b(m,\rho)$ is an edge; and let $Q^0_{m,\rho}$ denote the union of all triangles inside $Q^0$ for which $b(m,\rho)$ is an edge; finally leading to $R_{m,\rho} := U_{m,\rho} \cup Q_{m,\rho} \cup Q^0_{m,\rho}$, which is the union of all regions associated with the bond $b(m,\rho)$. 
    
    Clearly, if $b(m,\rho)\not\in \Gamma$, then there are two triangles in $U_{m,\rho}$, whereas $Q_{m,\rho} = Q^0_{m,\rho} = \emptyset$. If $b(m,\rho)\in \Gamma \setminus Q_0$, then there is one triangle in $U_{m,\rho}$, one rectangular strip in $Q_{m,\rho}$ and $Q^0_{m,\rho} = \emptyset$. Finally, if $b(m,\rho) \subset Q_0$, then there is one triangle in $U_{m,\rho}$, $Q_{m,\rho} = \emptyset$, and there is one triangle in $Q^0_{m,\rho}$. See Figure~\ref{fig2} for a visual intuition. 
    
    By construction, we thus get the equality
    \[
    \int_{\R^2\setminus \Gamma_0} \nabla u(x) \cdot \nabla Iv(x)dx = \frac12 \sum_{m \in \La}\sum_{\rho \in \Rc(m)} \left(\int_{R_{m,\rho}}\nabla_{\rho}u(x)dx\right)D_{\rho}v(m),
    \]
    whenever the left hand-side is well defined. Note that the $\frac12$ prefactor enters because of the double counting of bonds. 
    
    This is used in \cite{2018-antiplanecrack} to show that the predictor $\hat u_0$ satisfies 
    \begin{equation}\label{eqn-duhat0-hhat0}
    \sum_{m \in \La} D \hat u_0(m) \cdot Dv(m) = \sum_{m \in \La} \hat h_0(m) \cdot Dv(m),
    \end{equation}
    where $\hat h_0 \in \ell^2(\La; \R^4)$ is given by
    \begin{align*}
        (\hat h_0)_{\rho}(m) = \begin{cases}
            D_{\rho}\hat u_0(m) - \int_{R_{m, \rho}} \nabla_{\rho}\hat u_0(x)dx, \quad &b(m,\rho) \not\subset Q_0,\\
            D_{\rho}\hat u_0(m) - C_{m+\rho,m}, \quad &b(m,\rho) \subset Q_0.
        \end{cases}
    \end{align*}
    where the formulae for the finite $C_{m+\rho,m}$ are reported in \cite[Section~4.2.1.]{2018-antiplanecrack}. By construction, we also have that
    \begin{subequations} \label{eqn:hath_0-decay}
        \begin{align}
        b(m,\rho) \not\subset \Gamma &\implies |(\hat h_0(m))_{\rho}| \lesssim |\nabla^3 \hat u_0(m)| \lesssim |m|^{-5/2}\\
        b(m,\rho) \subset \Gamma\setminus Q_0 &\implies |(\hat h_0(m))_{\rho}| \lesssim |\nabla^2 \hat u_0(m)| \lesssim |m|^{-3/2}.
    \end{align}
    \end{subequations}
    The estimate for $\hat h_0$ when $b(m,\rho) \subset \Gamma \setminus Q_0$ turns out to be insufficient for our purposes. The estimate is sharp though and the way to improve the behaviour is to use the symmetrised interpolation instead, which we will now introduce. 
    
    Building upon the default interpolation operator $I$, we first introduce its mirrored version $I^{\rm mir}$, which, away from $\R^2 \setminus \Omega_{\Gamma}$, is defined via carving the squares into right-angle triangles across the other diagonal. Inside $\Omega_{\Gamma}$ we take $I^{\rm mir}$ to be the same as $I$. Naturally this operator has all the same properties as $I$ itself. See Figure~\ref{fig2} for a visual comparison between $I$ and $I^{\rm mir}$.
    
    This allows us to define a symmetric interpolation operator $I^{\rm sym}$ as
    \[
    I^{\rm sym}v := \frac{1}{2}(I v + I^{\rm mir} v).
    \]
    Another way of thinking about $I^{\rm sym}$ is to set the function to be the average of all four corners at the midpoint of each square and then interpolate affine on the 4 resulting triangles.
     
    Returning to \eqref{eqn-duhat0-hhat0}, we observe that using $I^{\rm sym}$ we obtain the improved estimate
    \begin{equation}\label{eqn-duhat0-hhat0-sym}
    \sum_{m \in \La} D \hat u_0(m) \cdot Dv(m) = \sum_{m \in \La} \hat h^{\rm sym}_0(m) \cdot Dv(m),
    \end{equation}
    where 
    \[
    (\hat h_0^{\rm sym}(m))_{\rho} = \frac{1}{2} \left((\hat h_0(m))_{\rho} + (\hat h_0^{\rm mir}(m))_{\rho}\right),
    \]
    for which we can prove the following. 
    \begin{lemma}\label{lem:symmetric_I_decay}
        The function $\hat h_0^{\rm sym}$ defined in \eqref{eqn-duhat0-hhat0-sym} satisfies
        \begin{equation}
            |\hat h_0^{\rm sym}(m)| \lesssim \lvert m \rvert^{-\frac{5}{2}}.
        \end{equation}
    \end{lemma}
    \begin{proof}
        For $b(m,\rho) \not\subset \Gamma$ the result follows directly from \eqref{eqn:hath_0-decay}. So suppose $b(m,\rho) \subset \Gamma \setminus Q_0$. By definition, we have
        \begin{align*}
            (\hat h_0^{\rm sym}(m))_{\rho} &= \frac{1}{2} \left((\hat h_0(m))_{\rho} + (\hat h_0^{\rm mir}(m))_{\rho}\right)\\
            &= D_\rho \hat{u}_0(m) - \frac{1}{2} \int_{U_{m\rho}} \nabla_\rho \hat{u}_0\,dx - \frac{1}{2} \int_{U_{m\rho}^{\rm mir}} \nabla_\rho \hat{u}_0\,dx - \int_{Q_{m\rho}} \nabla_\rho \hat{u}_0\,dx
        \end{align*}
        Taylor expanding the three integrals around their midpoints and using $\lvert \nabla^3 \hat{u} \rvert \lesssim \lvert m \rvert^{-\frac{5}{2}}$, we find 
        \[
       (\hat h_0^{\rm sym}(m))_{\rho} = \nabla_\rho \hat{u}_0(m + \frac{1}{2} \rho) - \frac{1}{4} \nabla_\rho \hat{u}_0(x_{U_{m\rho}}) - \frac{1}{4} \nabla_\rho \hat{u}_0(x_{U_{m\rho}^{\rm mir}})- \frac{1}{2} \nabla_\rho \hat{u}_0(x_{Q_{m\rho}}) + O(\lvert m \rvert^{-\frac{5}{2}}),\]
        where $x_A = \frac{1}{\lvert A \rvert} \int_A x \,dx$. The terms involving $\nabla^2 \hat{u}_0$ vanished due to these being the midpoints.
        
        Using the same idea again, we can Taylor expand each term in the weighted sum 
        \[
        \frac{1}{4} \nabla_\rho \hat{u}_0(x_{U_{m\rho}}) + \frac{1}{4} \nabla_\rho \hat{u}_0(x_{U_{m\rho}^{\rm mir}})+ \frac{1}{2} \nabla_\rho \hat{u}_0(x_{Q_{m\rho}})
        \]
        around the weighted midpoint
        \[ x_{m\rho} = \frac{1}{4} x_{U_{m\rho}} + \frac{1}{4} x_{U_{m\rho}^{\rm mir}}+ \frac{1}{2} x_{Q_{m\rho}}.\]
        Again, the terms involving $\nabla^2 \hat{u}_0$ cancel and we get
        \[
        (\hat h_0^{\rm sym}(m))_{\rho} = \nabla_\rho \hat{u}_0\big(m + \frac{1}{2} \rho\big) - \nabla_\rho \hat{u}_0(x_{m\rho}) + O(\lvert m \rvert^{-\frac{5}{2}}).
        \]
        Despite the symmetrisation, these two points are still different. Direct calculations give
        \begin{align*}
        x_{Q_{m\rho}} &=m + \frac{1}{2}\rho - \frac{1}{4} \sgn(m_2) e_2,\\
        \frac{1}{2} x_{U_{m\rho}^{\rm mir}} + \frac{1}{2} x_{U_{m\rho}}&=  m + \frac{1}{2}\rho + \frac{1}{3} \sgn(m_2) e_2,
        \end{align*}
        and thus
        \[x_{m\rho} =m + \frac{1}{2}\rho + \frac{1}{24} \sgn(m_2) e_2. \]
        In particular, we have
        \[(x_{m\rho})_1 = \left(m + \tfrac{1}{2} \rho\right)_1.\]
        With $y_{m\rho} = \frac{1}{2}(m + \frac{1}{2} \rho+x_{m\rho})=m + \frac{1}{2}\rho + \frac{1}{48} \sgn(m_2) e_2$, it follows that
        \begin{align*}
        (\hat h_0^{\rm sym}(m))_{\rho} &= -\frac{1}{24} \sgn(m_2) \partial_2 \nabla \hat{u}(y_{m\rho}) + O(\lvert m \rvert^{-\frac{5}{2}}).
        \end{align*}
        But now the boundary condition $\partial_2 \hat{u}_0 =0$ on $\Gamma_0$ implies that $\partial_2 \partial_1 \hat{u}_0((y_{m\rho})_1, 0^{\pm}) =0$. We thus also find
        \[\lvert \partial_2 \nabla_\rho \hat{u}(y_{m\rho}) \rvert= \lvert \partial_2 \nabla_\rho \hat{u}(y_{m\rho})-\partial_2 \nabla_\rho \hat{u}_0((y_{m\rho})_1, 0^{\pm}) \rvert \lesssim \lvert m \rvert^{-\frac{5}{2}}.\]
        Here, $0^{\pm}$ denotes the continuous extension on $\Gamma_0$ of $\hat u_0$ on the side corresponding to $\sgn(m_2)$.
    \end{proof}    
    
    \section{Proofs: The lattice Green function} \label{sec:proofsgreensfunction}
    \subsection{Setup and prerequisites}
    \begin{definition}\label{def-Ghom}
        A function $G_{\rm hom}\,\colon\, \La \to \R$ is said to be a \emph{homogeneous lattice Green's function} if 
        \[
        -{\rm Div}^{\rm hom} D^{\rm hom} G_{\rm hom}(m-\ell) = \delta(m,\ell),
        \]
        where $\delta(m,l)$ is the Kronecker delta and $D^{\rm hom} u(m) := (D_{\rho}u(m))_{\rho \in \Rc}$ (i.e. with no interaction bonds removed) and the same for ${\rm Div}^{\rm hom}$.
    \end{definition}
    \begin{theorem}[\protect{\cite[Section~6.2]{EOS2016}}]
        There exists a homogeneous lattice Green's function ${G_{\rm hom}\,\colon\, \La \to \R}$ in the sense of Definition~\ref{def-Ghom} and it satisfies a decay estimate
        \[
        |D^j G_{\rm hom}(m-\ell)| \lesssim (1+|m-\ell|^j)^{-1}.
        \]
    \end{theorem}
    To obtain zero-order results for the lattice Green's function in the anti-plane crack geometry, in \cite{2018-antiplanecrack} a discrete complex square root manifold is introduced to obtain decay results for atoms close to the crack surface. The full account is presented in \cite[Lemma~4.8, Case 2]{2018-antiplanecrack} (see also \cite{2020-near-crack-tip-plas} for a more in-depth discussion of this geometric framework with application to near crack-tip plasticity). In what follows, to simplify presentation, we use a more minimal presentation, for which the following definition is needed. 
    
    \begin{definition}[Reflected lattice point and function]\label{def-ref}
        Given a lattice point $m=(m_1,m_2) \in \La$, we define its reflected counterpart as $m_{\rm ref} := (m_1,-m_2) \in \La$. Likewise, given a lattice function $u\,\colon\, \La \to \R$, its reflected counterpart $u_{\rm ref}\,\colon\,\La \to \R$ is defined as, for $m = (m_1,m_2) \in \La$
        \[
            u_{\rm ref}(m) := \begin{cases}
                u(m_{\rm ref} )\quad &\text{if }\,m_1 < 0, m_2 > 0\\
                u(m)\quad &\text{otherwise}.
            \end{cases}
        \]
    \end{definition}
    
    Finally, we also note that the symmetric interpolation introduced in Section~\ref{sec:interpol} allows us to retrace the proof Lemma~\ref{lem:symmetric_I_decay} and arrive at the following useful result. 
    \begin{lemma}\label{lem-G_0-I-sym}
    The predictor $\hat G_0$ from \eqref{eqn-hatG0-formula} satisfies
    \[
    \sum_{m \in \La} D_m \hat G_0(m,s) \cdot Dv(m) - v(s) = \sum_{m \in \La} \hat g_0(m) \cdot Dv(m),
    \]
    where $|\hat g_0(m)| \lesssim |\nabla^3 \hat G_0(m,s)|$.
    \end{lemma}
    
    \subsection{Improved estimates for $\bar{G}_0$}\label{sec:improved-barG0}
    In order to prove Theorem~\ref{thm-Gbar1-decay}, we first need to establish auxiliary results about $\bar G_0$, partially improving upon the results established in \cite{2018-antiplanecrack}. Let us start with a suboptimal estimate for  $D_1 D_1 D_2 \bar{G}_0$ in an annulus $B_{5\lvert s \rvert/8}(0) \setminus B_{3\lvert s \rvert/8}(0)$.
    \begin{lemma}\label{lem:D1D1D2_Gbar0}
    \begin{equation}
    \lvert D_1 D_1 D_2 \bar{G}_0(\ell,s) \lvert \lesssim \lvert s \rvert^{-3+\delta}
    \end{equation}
    for any $\ell \in B_{5\lvert s \rvert/8}(0) \setminus B_{3\lvert s \rvert/8}(0)$.
    \end{lemma}
    \begin{proof}
    Let $\ell \in B_{5\lvert s \rvert/8}(0) \setminus B_{3\lvert s \rvert/8}(0)$ and choose the cutoff as in Definition~\ref{def-cutoff-fcts} with $c=\frac{1}{8}$. In particular, $\eta(m) \neq 0$ implies $\tfrac{\lvert s \rvert}{4} \leq \lvert m \rvert \leq \tfrac{3\lvert s \rvert}{4}$. Assume without loss of generality that $\ell_2 < 0$. In particular the reflected function $(\bar G_0)_{\rm ref}$, defined in Definition~\ref{def-ref} with reflection only in the first variable, satisfies 
    \[
    D_1 D_2(\bar G_0)_{\rm ref}(\ell,s) = D_1 D_2\bar G_0(\ell,s).
    \]
    Further, set $\sigma \in \Rc(m)$. We then have
    \begin{align*}
     D_{1 \sigma} D_1 D_2 \bar{G}_0(\ell,s) &= D_{1 \sigma} ( D_1 D_2 (\bar{G}_0)_{\rm ref}(\ell,s) \eta(\ell) )\\
     &= \sum_{m \in \La} -{\rm Div} \tilde D G_{\rm hom} (m-\ell) D_{1 \sigma} ( D_1 D_2 (\bar{G}_0)_{\rm ref}(m,s)\eta(m))\\
     &= \sum_{m\in \Lambda} \sum_{\rho \in \Rc} D_\rho D_{ -\sigma} G_{\rm hom} (m-\ell) D_{1\rho} (D_1 D_2 (\bar{G}_0)_{\rm ref}(m,s)\eta(m))\\
     &= \sum_{m \in \La} \sum_{\rho \in \Rc} D_{\rho} \big( \eta(m) D_{ -\sigma} G_{\rm hom} (m-\ell)\big) D_{1\rho}D_1 D_2 (\bar{G}_0)_{\rm ref}(m,s)\\
     &-\sum_{m \in \La}\sum_{\rho \in \Rc}D_1D_2(\bar G_0)_{\rm ref}(m,s)D_\rho\eta(m)D_{\rho} D_{-\sigma}G_{\rm hom}(m-l)\\
     &-\sum_{m \in \La}\sum_{\rho \in \Rc} D_{-\sigma}G_{\rm hom}(m-l) D_{\rho}\eta(m)D_{1\rho}(D_1D_2(\bar G_0)_{\rm ref}(m,s)\\
     &=:S_1 + S_2 + S_3,
    \end{align*}
    where we have used Lemma~\ref{lem:pushing_eta} in the last equality.
    
    To estimate $S_1$, we introduce short-hand notation $v(m):=\eta(m)D_{ -\sigma} G_{\rm hom}(m-l)$ and observe that by using the discrete manifold construction described in \cite[Section~4.3.3.]{2018-antiplanecrack},
    \[
    G_{\rm ref} = (\hat G_0)_{\rm ref} + (\bar G_0)_{\rm ref}
    \]
    satisfies a manifold equivalent of the discrete PDE given in \eqref{eqn-G-pde}. With the support of $v$ bounded away both from the origin and from the new artificial cut at $m_1=0$, $m_2>0$. We can thus rewrite $S_1$ as follows.
    \begin{align*}
    S_1 &= \sum_{m\in\La}\sum_{\rho \in \Rc(m)}D_{1\rho}D_1D_2(\bar G_0)_{\rm ref}(m,s)D_{\rho}v(m)\\
    &= - \sum_{m\in\La}\sum_{\rho \in \Rc(m)}D_{1\rho}D_1D_2(\hat G_0)_{\rm ref}(m,s)D_{\rho}v(m) + DDv(s) = \\
    &= \sum_{b(m,\rho) \not\subset \Gamma} g^{\rm R}_{11}(m,\rho)D_{\rho}v(m) + \sum_{b(m,\rho)\subset \Gamma} h^{\rm R}_{11}(m,\rho)D_{\rho}v(m),
    \end{align*}
    where, for $m \in {\rm supp}\,\eta$, we have $|g^{\rm R}_{11}(m,\rho)| \lesssim |s|^{-5}$ and $|h^{\rm R}_{11}(m,\rho)| \lesssim |s|^{-4}$ are the terms obtained from introduced a single interpolation (see Section~\ref{sec:interpol}). The decay of $|Dv(m)|$ is thus enough to conclude that 
    \[
    |S_1| \lesssim |s|^{-4}.
    \]
    
    Note by using symmetric interpolation as described in Section~\ref{sec:interpol}, we can in fact improve this estimate to
    \[
        |S_1| \lesssim |s|^{-5}\log \lvert s \rvert,
    \]
    but this will not improve the overall result, as we are restricted by $S_2$ and $S_3$ anyway. 
    
    For $S_2$ we can directly estimate all terms pointwise for $m \in {\rm supp}\,D\eta$, allowing us to conclude that
    \[
    |S_2| \lesssim \sum_{m \in \La} \lvert D \eta(m) \rvert \lvert D^2G_{\rm hom}(m-\ell) \rvert \lvert D_1 D_2 (\bar{G}_0)_{\rm ref}(m,s) \rvert \lesssim \lvert s \rvert^{-3+\delta/2}.
    \]
    For $S_3$, we sum $D_\rho$ by parts to obtain
    \begin{align*}
    S_3 &= -\sum_{m \in \La}\sum_{\rho \in \Rc} D_{-\sigma}G_{\rm hom}(m-l) D_{\rho}\eta(m)D_{1\rho}(D_1D_2(\bar G_0)_{\rm ref}(m,s)\\
    &= - \sum_{m \in \La} -{\rm Div}(D_{-\sigma}G_{\rm hom}(m-l)D\eta(m)) D_1 D_2(\bar G_0)_{\rm ref}(m,s) \\
    &= - \sum_{m \in \La} -{\rm Div}D\eta(m) D_{-\sigma}G_{\rm hom}(m-l)D_1D_2(\bar G_0)_{\rm ref}(m,s)\\
    &- \sum_{m \in \La}\sum_{\rho \in \Rc}D_{-\rho}D_{-\sigma}G_{\rm hom}(m-l)D_{\rho}\eta(m-\rho)D_1 D_2(\bar G_0)_{\rm ref}(m,s).
    \end{align*}
    
    It thus follows from another direct pointwise estimate that 
    \begin{align*}
        \lvert S_3 \rvert &\lesssim \sum_{m \in \La} \lvert D^2 \eta(m) \rvert \lvert DG_{\rm hom}(m-\ell) \rvert \lvert D_1 D_2 (\bar{G}_0)_{\rm ref}(m,s)\rvert\\
        &+\sum_{m \in \La}\lvert D \eta(m) \rvert \lvert D^2G_{\rm hom}(m-\ell) \rvert \lvert D_1 D_2 (\bar{G}_0)_{\rm ref}(m,s) \rvert\\
        &\lesssim \lvert s \rvert^{-3+\delta/2} \nonumber
    \end{align*}
    \end{proof}
    
    Using Lemma~\ref{lem:D1D1D2_Gbar0}, we will now obtain an improved estimates on $D_1 D_2 \bar{G}_0$ in an annulus $B_{\tfrac{\lvert s \rvert}{2}}(0) \setminus B_{\tfrac{\lvert s \rvert}4}(0)$.
    \begin{lemma} \label{lem:Gbaroldannulus}
    It holds that 
    \begin{equation} \label{eq:Gbaroldannulus}
        \lvert D_1 D_2 \bar{G}_0(\ell,s) \rvert \lesssim \lvert s\rvert^{-3+\delta}.
    \end{equation}
    for $\lvert s \rvert/4 \leq \lvert \ell \rvert \leq \lvert s \rvert/2$.
    \end{lemma}
    
    \begin{proof}
    Set $v(m):= D_2G(m,l)$. We have
    \begin{align}
        D_1 D_2 \bar{G}_0(\ell,s) &= \sum_{m \in \La} \sum_{\rho \in \Rc(m)} D_{1\rho} D_2 \bar{G}_0(m,s) D_{\rho}v(m) \nonumber\\
        &= \underbrace{\sum_{m \in \La}\sum_{\rho \in \Rc(m)} - D_{\rho} D_2 \hat{G}_0(m,s) D_{\rho} v(m)}_{=: T_1} + \underbrace{\vphantom{\sum_{\rho \in \Rc(m)} }Dv(s)}_{=:T_2}\label{eqn:d1d2Gbar0}.
    \end{align}
    To elucidate the idea of proving this result by employing an appropriate cut-off function, followed by summation / integration by parts, we begin by recalling Definition~\ref{def-cutoff-fcts} and introducing a continuum cut-off function $\eta \,\colon\, \R^2 \to \R$ given by 
    \[
    \eta(x) := \hat{\eta} \left(\frac{\lvert x-s\rvert}{|s|/4}\right), \]
    for which it holds that 
    \[
    {\rm supp}D\eta = B_{s/4}(s) \setminus B_{s/8}(s), \; \eta(x) \neq 0 \implies \frac{3|s|}{4} < \lvert x \rvert < \frac{5|s|}{4}.
    \] 
    This allows us to rewrite $T_2$ as follows. We have, due to the PDE that $\hat G_0$ satisfies, 
    \begin{align*}
        Dv(s) &= \underbrace{2 \int_{\R^2 \setminus \Gamma_0}\left( \nabla(D_2 \hat G_0(x,s)\eta(x)) \cdot \nabla I v(x)\right) dx}_{:= T_{2,1}}\\ 
        &+ \underbrace{2 \int_{\R^2 \setminus \Gamma_0}\left( \nabla(D_2 \hat G_0(x,s)(1-\eta(x))) \cdot \nabla I v(x)\right) dx}_{:= T_{2,2}}, 
    \end{align*}
    where we use the symmetric interpolation operator $I^{\rm sym}$ introduced in Section~\ref{sec:interpol}, but abuse the notation by identifying $I \equiv I^{\rm sym}$ for brevity. 
    
    Let us focus on $T_{2,1}$. By adding and subtracting the same term, we obtain
    \begin{align*}
        T_{2,1} &= 2\sum_{i=1,2} \int_{\R^2 \setminus \Gamma_0} \partial_i(D_2 \hat G_0(x,s)\eta(x))\left(\partial_i Iv(x)\right)dx \\ 
        &= 2\sum_{i=1,2} \int_{\R^2 \setminus \Gamma_0} \partial_i(D_2 \hat G_0(x,s)\eta(x))\left(\partial_i Iv\left(s + \frac12 e_i\right)\right) dx \\ 
        &+ 2\sum_{i=1,2} \int_{\R^2 \setminus \Gamma_0} \partial_i (D_2 \hat G_0(x,s)\eta(x))\left(\partial_i Iv(x) - \partial_i Iv\left(s + \frac12 e_i\right)\right) dx \\
        &:= T_{2,1,1} + T_{2,1,2}.
    \end{align*}
    Note that $s + \tfrac 12 e_i$ is the midpoint of a lattice bond. We then apply integration by parts to $T_{2,1,1}$ and observe that 
    \begin{align*}
        T_{2,1,1} &= 2\sum_{i=1,2} \partial_i Iv\left(s + \frac12 e_i \right)\left(\int_{\R^2 \setminus \Gamma_0} \partial_i(D_2 \hat G_0(x,s) \eta(x)) dx\right)\\
        &+ 2\partial_1 Iv\left(s + \frac12 e_1 \right)\int_{\Gamma_0} D_2 \hat G_0(x,s)\eta(x) dx.
    \end{align*}

    Recalling the discussion in Section~\ref{sec:interpol} about the different regions associated with the interpolation operators, we note that by construction and since the support of $\eta$ is bounded away from $Q_0$, we have
    \begin{align*}
        T_{2,1,1} &= \sum_{b(m,\rho) \subset \Gamma}D_{\rho}v(s)\int_{\overline{Q_{m,\rho}}\cap \Gamma_0} D_2 \hat G_0(x,s) \eta(x)dx
    \end{align*}
    and similarly
    \begin{align*}
        T_{2,1,2} &= \sum_{m\in \La}\sum_{\rho \in \Rc(m)} (D_{\rho}v(m)-D_{\rho}v(s))\left(\int_{R_{m,\rho}} \nabla_{\rho}(D_2 \hat G_0(x,s)\eta(x))dx\right)
    \end{align*}
    
    The same reasoning applies to rewriting the other term in \eqref{eqn:d1d2Gbar0}, namely
    \begin{align*}
        T_1 = &= \sum_{m \in \La}\sum_{\rho \in \Rc(m)} - D_{\rho}(D_2 \hat{G}_0(m,s)\eta(m))D_{\rho} v(m)\\
        &+ \sum_{m \in \La}\sum_{\rho \in \Rc(m)} - D_{\rho}(D_2 \hat{G}_0(m,s)(1-\eta(m)))D_{\rho}v(m)\\
        &=: T_{1,1} + T_{1,2}. 
    \end{align*}
    Focusing for now on $T_{1,1}$, we note that
    \begin{align*}
        T_{1,1} &= \underbrace{\sum_{m \in \La}\sum_{\rho \in \Rc(m)} - D_{\rho}(D_2 \hat{G}_0(m,s)\eta(m))D_{\rho} v(s)}_{=: T_{1,1,1}}\\  &+ \underbrace{\sum_{m \in \La}\sum_{\rho \in \Rc(m)} - D_{\rho}(D_2 \hat{G}_0(m,s)\eta(m)) \left(D_{\rho} v(m)-D_{\rho} v(s) \right)}_{=:T_{1,1,2}}
    \end{align*}
    and further that it follows from summation by parts that  
    \begin{align*}
        T_{1,1,1} = \sum_{b(m,\rho) \subset \Gamma}-D_{\rho}v(s) D_2 \hat G_0(m,s)\eta(m).
    \end{align*}
    
    Recalling from \eqref{eqn:d1d2Gbar0} that
    \[
    D_1 D_2 \bar G_0(\ell,s) = T_1 + T_2,
    \]
    we sum together $T_{1,1,1}$ and $T_{2,1,1}$ to obtain
    \begin{align*}
        T_{1,1,1} + T_{2,1,1} = \sum_{b(m,\rho) \subset \Gamma}-D_{\rho}v(s)\left(D_2 \hat G_0(m,s) \eta(m) - \int_{\overline{Q_{m,\rho}}\cap \Gamma_0}D_2 \hat G_0(x,s) \eta(x) dx \right).
    \end{align*}
    Since $|\overline{Q_{m,\rho}} \cap \Gamma_0| = 1$, a first order quadrature estimate ensures that
    \begin{align*}
        |T_{1,1,1} + T_{2,1,1}| &\lesssim \sum_{b(m,\rho) \subset \Gamma}|Dv(s)|\left(|D_1 D_2 \hat G_0(m,s)| + |D_2 \hat G_0(m,s)||D\eta(m)|\right)\\
        &\lesssim |s||s|^{-2}\left(|s|^{-2} + |s|^{-1-1}\right) = |s|^{-3}.
    \end{align*}
    
    Similarly we can sum together $T_{1,1,2}$ and $T_{2,1,2}$ to obtain 
    \begin{align}\label{eqn:T112T212}
        T_{1,1,2} + T_{2,1,2} = \sum_{m \in \La}\sum_{\rho \in \Rc(m)}\Bigg( & \left( \int_{R_{m,\rho}} \nabla_{\rho}(D_2 \hat G_0(x,s)\eta(x))dx\right)\\  &- D_{\rho}(D_2 \hat{G}_0(m,s)\eta(m))\Bigg)\left(D_{\rho} v(m)-D_{\rho} v(s) \right)\nonumber
    \end{align}
    The mid-point quadrature estimate, together with the symmetrisation trick, as established in Lemma~\ref{lem-G_0-I-sym}, implies that
    \begin{align*}
    |T_{1,1,2} + T_{2,1,2}| &\lesssim \sum_{m \in \La}|\nabla^3(D_2 \hat G_0(m,s) \eta(m))||D_{\rho}v(m)-D_{\rho}v(s)|\\
    &\lesssim \sum_{m \in {\rm supp} \eta}|m-s||\nabla^3(D_2 \hat G_0(m,s) \eta(m))|\sup_{m \in {\rm supp}\eta}|D_1^2D_2G(m,l)|\\
    &\lesssim |s|^{-3} \sum_{m \in {\rm supp }\eta }|m-s||\nabla^3D_2 \hat G_0(m,s)| +  \underbrace{|s|^{-3} \sum_{m \in {\rm supp }\eta}|m-s||s|^{-3}}_{\lesssim |s|^{-3}}. 
    \end{align*}
    And we observe that
    \begin{align*}
    &|s|^{-3} \sum_{m \in {\rm supp }\eta }|m-s||\nabla^3D_2 \hat G_0(m,s)| \lesssim \\
    &\lesssim \lvert s \rvert^{-3+\delta} \sum_{m \in {\rm supp } \eta } |m-s| \big( (1+\lvert \omega_m \rvert^5 \lvert \omega_s \rvert^1 \lvert \omega_{m s}^{-} \rvert^2)^{-1} + (1+\lvert \omega_m \rvert^3 \lvert \omega_s \rvert^1 \lvert \omega_{m s}^{-} \rvert^4)^{-1} \big) \\
    &\lesssim \lvert s \rvert^{-3+\delta} \sum_{m \in {\rm supp } \eta, m \neq s } |m-s| \big( s^{-3} \lvert m-s \rvert^{-2} \lvert \omega_{m s}^{+} \rvert^{2}  + s^{-2}\lvert m-s \rvert^{-4}\lvert \omega_{m s}^{+} \rvert^{4}  \big)\\
    &\lesssim \lvert s \rvert^{-3+\delta}  \left(\lvert s\rvert^{-2 }\sum_{\substack{m \in {\rm supp } \eta,\\ m \neq s} } |m-s|^{-1} + \sum_{\substack{m \in {\rm supp} \eta,\\ m \neq s}} |m-s|^{-3}\right)\\
    &\lesssim \lvert s \rvert^{-3+\delta}  (\lvert s\rvert^{-2 } \lvert s \rvert + 1)\\
    &\lesssim \lvert s \rvert^{-3+\delta}
    \end{align*}
    
    Finally, we turn our attention to terms $T_{1,2}$ and $T_{2,2}$. Replicating the argument that allows to arrive at \eqref{eqn:T112T212}, we observe that
    \begin{align*}
        T_{1,2} + T_{2,2} = \sum_{m \in \La}\sum_{\rho \in \Rc(m)}\Bigg(& \left(\int_{R_{m,\rho}} \nabla_{\rho}(D_2 \hat G_0(x,s)(1-\eta(x)))dx\right)\\ &- D_{\rho}(D_2 \hat{G}_0(m,s)(1-\eta(m)))\Bigg)D_{\rho} v(m)
    \end{align*}
    and hence
    \begin{align*}
    |T_{1,2} + T_{2,2}| &\lesssim \sum_{m \in \La}|\nabla^3(D_2 \hat G_0(m,s)(1- \eta(m)))||D_{1}D_2G(m,l)|\\
     &\lesssim \sum_{m \in \La}\underbrace{|\nabla^3D_2 \hat G_0(m,s)||(1- \eta(m))||D_{1}D_2G(m,l)|}_{=:S(m)} + |s|^{-4+\delta},
    \end{align*}
    where we have used the fact that if at least one gradient falls onto $(1-\eta(m))$, every term scales fully with $|s|$, so we can take them out and add volume term to arrive at, for $j=1,2,3$, 
    \[
    \sum_{m \in \La}|\nabla^{3-j}D_2 \hat G_0(m,s)|\nabla^{j}(1- \eta(m))|||D_{1}D_2G(m,l)| \lesssim |s|^{-6+\delta} |{\rm supp} D\eta| \lesssim |s|^{-4+\delta}.
    \]
    To estimate the remaining term, we first note that 
    \[
    \sum_{m \in B_{|s|/8}(0)}S(m) \lesssim \sum_{m \in B_{|s|/8}(0)}(1+|\omega_m|^5|\omega_s||\omega_{ms}^-|^2)^{-1}(1+|\omega_m||\omega_\ell||\omega_{m\ell}^-|^{2-\delta})^{-1} \lesssim |s|^{-3+\delta}
    \]
    and 
    \[
    \sum_{m \in B_{|\ell|/16}(\ell)}S(m) \lesssim |s|^{-4}\left(1+ \int_{\varepsilon}^{|l|}r^{-1+\delta} dr \right) \lesssim |s|^{-4+\delta}.
    \]
    Outside of these two regions and with the region near $|s|$ excluded via $1-\eta$, the remaining sum can be estimated as
    \[
    \sum_{m \in \La \setminus B_{|s|/8}(0) \cup B_{|\ell|/16}(\ell)} S(m) \lesssim |s|^{-1}\int_{|s|}^{\infty} r^{(-10+\delta)/2}dr \lesssim |s|^{-4+\delta},
    \]
    which allows to conclude that
    \[
    |T_{1,2} + T_{2,2}| \lesssim |s|^{-3+\delta},
    \]
    thus ensuring that 
    \[
    |D_1 D_2 \bar{G}_0(l,s)|\lesssim |s|^{-3+\delta},
    \]
    which is what we set out to prove. 
    \end{proof}
    
    We need one more small Lemma before we actually go and improve beyond $\hat G_0$.
    \begin{lemma} \label{lem:DG_decay}
        For all $\ell, s \in \La$, the full lattice Green's function $G = \hat G_0 + \bar G_0$, defined in \eqref{eqn-G-old-decomp}, satisfies
        \[\lvert D_1 G(\ell,s)\rvert \lesssim (1+ \lvert \omega(\ell) \rvert \lvert\omega(\ell) - \omega(s) \rvert^{1-\delta})^{-1}. \]
    \end{lemma}
    \begin{proof}
        First note that this holds true (even with $\delta =0$) for $\hat G_0$ instead of $G$ according to \cite[Lemma 4.4]{2018-antiplanecrack}.
    
        Furthermore, if $\lvert \ell \rvert > \frac{\lvert s \rvert}{2}$ we can apply \cite[Lemma 4.8]{2018-antiplanecrack} to get the same for $\bar G_0$ and thus conclude
        \[\lvert D_1 G(\ell,s) \rvert \lesssim (1+\lvert \omega_\ell \rvert \lvert \omega_{\ell s}^{-} \rvert)^{-1} \]
        there.
    
        We are left with the case $2\lvert \ell \rvert \leq \lvert s \rvert$. Consider $v(s) = D_1G(\ell,s)$ as a function in $s$. Then we know that
        $\lvert Dv(s) \rvert \lesssim (1+\lvert \omega(\ell) \rvert \lvert \omega(s) \rvert (\lvert \omega(\ell) - \omega(s)\rvert)^{2-\delta} )^{-1}$ according to \cite[Theorem 2.6]{2018-antiplanecrack}. That means $\lvert Dv(s) \rvert \lesssim \lvert \ell \rvert^{-1/2} \lvert s \rvert^{-3/2+\delta/2}$ for $\lvert s \rvert \geq 2\lvert \ell \rvert$.
    
        We can now apply Lemma \ref{lem:decay_with_sobolev} to see that $v$ has a limit $v_{\infty}= v_{\infty}(\ell)$ at infinity with
        \[\lvert v(s) -  v_{\infty} \rvert \lesssim \lvert \ell \rvert^{-1/2} \lvert s \rvert^{-1/2+\delta/2}.\]
        Note that this estimate was just based on a telescope sum to infinity and thus the constants do not depend on $\ell$. Also note that for $\lvert s \rvert \geq 2\lvert \ell \rvert$ we can estimate
        \[\lvert \ell \rvert^{-1/2} \lvert s \rvert^{-1/2+\delta/2} \leq \lesssim (1+ \lvert \omega(\ell) \rvert \omega(\ell) - \omega(s) \rvert^{1-\delta})^{-1}.\]
        Hence we showed that there is a $v_{\infty}(\ell)$ such that 
        \[\lvert D_1 G(\ell,s) - v_{\infty}(\ell) \rvert \lesssim (1+ \lvert \omega(\ell) \rvert \omega(\ell) - \omega(s) \rvert^{1-\delta})^{-1}. \]
        However, $v_\infty$ is just the limit $\lim_{\ell \to \infty} D_1 G(\ell,s)$. Clearly this is $0$ for $\hat G_0$, but we also have $D_1 \bar G_0 (\cdot,s) \in \ell^2$ by definition in \cite{2018-antiplanecrack}, hence $v_\infty =0$ overall. 
    \end{proof}
    \subsection{Proof of Proposition~\ref{prop-ghat1}}\label{sec:proof-prop-ghat1}
    \begin{proof}
    Using the ansatz \eqref{eqn-hat-G1_decomp}, we see that $\hat G_1^m$ has to satisfy the discrete PDE
    \[
    -{\rm Div}D \hat G^m_1(m) = {\rm Div}\omega_2(m).
    \]
    Remarkably, this is the same pointwise discrete PDE as for $\bar u_0$, when the pair-potential is quadratic, that is $\phi_{\rm quad}(r) = \frac12r^2$. Thus the existence of $\hat G_1^m$ follows directly from Theorem~\ref{thm:ubar0}, as well as the estimate
    \[
        |D \hat G^m_1(\ell)| \lesssim |\ell|^{-3/2 + \delta}.
    \]
    In particular, the quadratic nature of $\phi_{\rm quad}$ ensures that, unlike for a general nonconvex $\phi$ as in \cite{2018-antiplanecrack}, $\hat G_1^m$ does not need a small pre-factor in front.
    
    As $3/2-\delta >1$, we can then use Lemma~\ref{lem:decay_with_sobolev} to also get the estimate
    \[
        |\hat G^m_1(\ell)| \lesssim |\ell|^{-1/2 + \delta}.
    \]
    after shifting $G^m_1$ by a constant.
    \end{proof}
    \subsection{Estimates for $\bar G_1$}
    Recall that we work with a decomposition of the full lattice Green's function $G$ given by
    \begin{align*}
    G &= \hat G_0 + \bar G_0\\
    &=\hat G_0 + \hat G_1 + \bar G_1,
    \end{align*}
    and our aim is to estimate $|D_1 D_2 \bar G_1|$ in the vicinity of the crack tip. To this end, we want to leverage the technical results of on the improved decay of $\bar G_0$ from Section~\ref{sec:improved-barG0} and the properties of the discrete geometry predictor $\hat G_1$ from Proposition~\ref{prop-ghat1}.
    
    It will turn out that the following auxiliary result, which extends \cite[Lemma~4.10]{2018-antiplanecrack} for the case when $m \in \Gamma$,  will prove useful. 
    \begin{lemma}\label{lem:Hf-Gamma0}
        It holds that when $m \in \Gamma$ then
        \[
        |{\rm Div}D \omega_2 (m)| \lesssim |\nabla^4 \omega_2(m)|
        \]
        and, if $s \not = m$, 
        \[
        |{\rm Div}_1 D_1 D_2 \hat G_0(m,s)| \lesssim |\nabla^4_m \nabla_s \hat G_0(m,s)|.
        \]
    \end{lemma}
    \begin{proof}
        Let $f$ denote either $f = \omega_2$ or $f = D_2 \hat G_0(\cdot,s)$, which in particular implies that, for each $m \in \La$, $m \neq s$, we have $\Delta f(m) = 0$, and further $\nabla f \cdot \nu = 0$ on $\Gamma_0$.
        
        As established in \cite[Lemma~4.10]{2018-antiplanecrack}, if $m \in \pm \Gamma$, then
        \begin{align*}
            {\rm Div} Df(m) &= 2 \sum_{\rho \in \Rc(m)} D_{\rho}f(m) = \pm 2 \nabla f(m) \cdot e_2 + 2\Delta f(m) - \nabla^2 f(m)[e_2]^2 \pm \frac13 \nabla^3 f(m)[e_2]^3\\ &+ \mathcal{O}(|\nabla^4 f(m)|)
        \end{align*}
        We now define $m_0 := m \mp \frac12 e_2 \in \Gamma_0$, and Taylor-expand $f$ around $m_0$, namely
        \begin{align*}
        f(m) &= f(m_0) + \nabla f(m_0) \cdot (m-m_0) + \frac12 \nabla^2 f(m)[m-m_0]^2 + \frac16 \nabla^3 f(m_0)[m-m_0]^3 + \mathcal{O}(|\nabla^4 f(m)|)\\
        &= f\left(m\mp\frac12 e_2\right) \pm \frac12\partial_2 f\left(m\mp \frac12e_2\right) + \frac18 \partial^2_{22} f\left(m\mp \frac12 e_2\right) \pm \frac1{48} \partial^3_{222}f\left(m\mp \frac12\right)\\ &+ \mathcal{O}(|\nabla^4 f(m)|)
        \end{align*}
        Combining the two Taylor expansions and the PDE that $f$ satisfies, we arrive at
        \[
            -{\rm Div}Df(m) = -\frac1{12} \partial^3_{222} f\left(m + \frac12 e_2\right) + \mathcal{O}(|\nabla^4 f(m)|).
        \]
        The results thus follows from the readily verifiable fact that, in both cases, $\partial^3_{222}f(m_0) = 0$ for all $m_0 \in \Gamma_0$.
    \end{proof}

    We begin by considering the residual, given by $-{\rm Div_1}D_1D_2 \bar G_1$, where ${\rm Div}_1$ denotes the divergence operator being applied with respect to the first variable. However, it turns out that the correction $\hat G_1$ is only very good for small $m$. Because of that we introduce a cut off variant
    \begin{align}\label{eqn-def-mu}
        \hat G_{1,{\mu}}(m,s) := \hat G_1(m,s) \mu(m,s), \quad \mu (m,s) := \hat \eta\Big(\frac{\lvert m \rvert }{\lvert s \rvert^{1/2}}\Big),
    \end{align}
    where $\hat \eta$ is the scalar cut-off function from Definition~\ref{def-cutoff-fcts}, which in particular implies that $\hat G_{1,\mu}$ coinicides with $\hat G_1$ when $|m| \leq \frac12 |s|^{\frac12}$ and is identically zero for $|m| \geq |s|^{1/2}$. Note that, in principle, instead of $|s|^{1/2}$, one could choose a different cut-off radius, e.g. $|s|^{\alpha}$ and optimize over $\alpha$. It turns out that $\alpha = \frac12$ is optimal, see Remark~\ref{rem:exponent-choice} below for a sketch of the argument.
    
    With this choice we have the following residual estimate.
    
    \begin{lemma}\label{lem:HD2Gbar1cut}
    For $\lvert m \rvert \leq \frac{\lvert s \rvert}{2}$ we have
    \[
        \lvert {\rm Div}_1 D_1 D_2 \bar G_{1,{\mu}}(m,s) \rvert \lesssim \left\{\begin{array}{lr}
            \lvert m \rvert^{-1/2 + \delta}\lvert s \rvert^{-5/2}, & \text{for } \lvert m \rvert \leq \lvert s \rvert^{1/2},\\
            \lvert m \rvert^{-7/2}\lvert s \rvert^{-3/2}, & \text{for } \lvert m \rvert > \lvert s \rvert^{1/2}.
            \end{array}\right.
      \]
    \end{lemma}
    \begin{proof}
        The starting point of the proof is the identity
        \begin{equation}\label{eqn:HD2barG1mu}
        {\rm Div}_1 D_1 D_2 \bar G_{1,\mu}(m,s) = -{\rm Div}_1 D_1 D_2 \hat G_{0}(m,s) - {\rm Div}_1 D_1 D_2 \hat G_{1,\mu}(m,s). 
        \end{equation}
        We recall from \eqref{eqn-hatG0-formula} that, with $F(x) = -C \log \lvert x \rvert$ and $C=\frac{1}{4\pi}$, we have
        \[
        \hat G_0(m,s) = F(\omega(s)-\omega(m)) + F(\omega^\ast(s)-\omega(m)).
        \]
        Taylor expansions of $F$ around $\omega(s)$ and $\omega^\ast(s)$ result in
        \begin{subequations}\label{eqn-Ghat0-decomp}
        \begin{align}
            \hat G_0(m,s) &= -(\nabla F(\omega(s))+ \nabla F(\omega^\ast(s))\omega(m)\\
            &+ \frac{1}{2}(\nabla^2 F(\omega(s))+ \nabla^2 F(\omega^\ast(s))[\omega(m)]^2\\
            &- \frac{1}{2}\int_0^1 (1-t)^2 (\nabla^3 F(\omega(s)-t\omega(m)) + \nabla^3 F(\omega^\ast(s)-t\omega(m)))[\omega(m)]^3 \,dt
        \end{align}
        \end{subequations}
    
        Let's take the discrete difference and look at the different terms:
        \[
        -{\rm Div}_1 D_1D_2 \hat G_0(m,s) = T_1 + T_2 + T_3.
        \]
        First note that $T_1 = {\rm Div}_1 D_1 D_2 \hat G_1(m,s)$. However, we have to consider $\mu$. Let us get back to that and consider the other two terms first.   
        
        For $T_2$ we see that
        \[
        \nabla^2F(x)[v]^2 = -C\left(\frac{\lvert v \rvert^2}{\lvert x \rvert^2} -2 \frac{(x \cdot v)^2}{\lvert x \rvert^4}\right).
        \]
        Hence,
        \[
        T_2 = C {\rm Div}_1 D_1 D_2 \Big(\frac{\lvert m \rvert}{\lvert s \rvert} - \frac{(\omega(m) \cdot \omega(s))^2}{\lvert s \rvert^2} - \frac{(\omega^*(s) \cdot \omega(m))^2}{\lvert s \rvert^2}  \Big).
        \]
        Recalling that 
        \[
        \omega(m) = |m|^{1/2} (\cos(\theta_m/2), \sin(\theta_m/2)), \; \omega(s) = |s|^{1/2} (\cos(\theta_s/2), \sin(\theta_s/2)), \; \omega^*(s) = (-\omega_1(s),\omega_2(s)),
        \]
        and using the dot product formula $x \cdot v = |x||v|\cos \gamma$, where $\gamma$ is the angle between $x$ and $v$, we have
        \begin{align*}
            (\omega(m)\cdot \omega(s))^2&= \lvert m \rvert \lvert s \rvert (\cos(\theta_m/2)\cos(\theta_s/2)+ \sin(\theta_m/2)\sin(\theta_s/2))^2\\
            (\omega(m)\cdot \omega^\ast(s))^2&= \lvert m \rvert \lvert s \rvert (-\cos(\theta_m/2)\cos(\theta_s/2)+ \sin(\theta_m/2)\sin(\theta_s/2))^2.
        \end{align*}
        Expanding the squares and using standard trigonometric identities, we arrive at 
        \begin{align*}
            T_2 &= C {\rm Div}_1 D_1 D_2 \frac{\lvert m \rvert}{\lvert s \rvert}\Big( 1- 2 \cos^2(\theta_m/2)\cos^2(\theta_s/2) - 2 \sin^2(\theta_m/2)\sin^2(\theta_s/2) \Big)\\
            &= -C {\rm Div}_1D_1 D_2 \Big(\frac{\lvert m \rvert}{\lvert s \rvert}\cos(\theta_m)\cos(\theta_s)\Big)\\
            &= -C {\rm Div}_1 D_1 D_2 \frac{s_1 m_1}{\vert s \rvert^2}\\
            &=0,
        \end{align*}
        since $-{\rm Div}_1 D_1 m_1 =0$. Note that this remains true even if $m \in \Gamma$.
        
        That leaves $T_3$. With
        \begin{align}
        D_s &(\nabla^3 F(\omega(s)-t\omega(m)) + \nabla^3 F(\omega^\ast(s)-t\omega(m)))_\rho \nonumber \\
        &= \int_0^1 \nabla^4 F(\omega(s+\tau \rho)-t\omega(m))[\nabla \omega (s+\tau \rho)\rho] + \nabla^4 F(\omega^\ast (s+ \tau \rho)-t\omega(m))[\nabla \omega^\ast (s+\tau \rho)\rho]\,d\tau\nonumber \\
        &:= A(m,s,t)_\rho, \label{eq:GreensTaylorError}
        \end{align}
        we have to estimate
        \[T_3 = \frac{1}{2}\int_0^1 (1-t)^2  {\rm Div}_1 D_1 \Big( A(m,s,t) [\omega(m)]^3\Big)\,dt.\]
        For this, we first estimate \eqref{eq:GreensTaylorError} by itself and with discrete differences in $m$ as
        \begin{align*}
            \lvert A(m,s,t) \rvert &\lesssim \lvert s \rvert^{-4/2} \lvert s \rvert^{-1/2} = \lvert s \rvert^{-5/2}, \\
            \lvert D_1 A(m,s,t) \rvert &\lesssim \lvert s \rvert^{-5/2} \lvert s \rvert^{-1/2} \lvert m \rvert^{-1/2} = \lvert s \rvert^{-3}\lvert m \rvert^{-1/2}, \\
            \lvert D_1^2 A(m,s,t) \rvert &\lesssim \lvert s \rvert^{-6/2} \lvert s \rvert^{-1/2} \lvert m \rvert^{-2/2} = \lvert s \rvert^{-7/2}\lvert m \rvert^{-1}.
        \end{align*}
        Using a product rule for discrete difference, which holds true if one includes extra shifts, we thus find
        \begin{align*}
            \lvert T_3 \rvert &\lesssim \sup_t \Big\lvert{\rm Div}_1 D_1 \Big( A(m,s,t) [\omega(m)]^3 \Big\rvert \\
            &\lesssim \lvert m \rvert^{-1/2} \sup_t \lvert A \rvert + \lvert m \rvert^{1/2}  \sup_t \lvert D_1 A \rvert + \lvert m \rvert^{3/2} \sup_t \lvert D_1^2 A \rvert \\
            &\lesssim \lvert s \rvert^{-5/2} \lvert m \rvert^{-1/2} + \lvert s \rvert^{-3} + \lvert s \rvert^{-7/2} \lvert m \rvert^{1/2}\\
            &\lesssim \lvert s \rvert^{-5/2} \lvert m \rvert^{-1/2}.
        \end{align*}
        Note that this argument also works when $m \in \Gamma$, because in \eqref{eqn-Ghat0-decomp}, on the left-hand side we have $\hat G_0$, which satisfies $\nabla_{m_0} \hat G_0(m_0,s) \cdot e_2 = 0$ when $m_0 \in \Gamma_0$ and the same is true for the first two terms on the right-hand side and so the same must be true for the third term. As a result, as in Lemma~\ref{lem:Hf-Gamma0}, we also always have two derivatives to distribute even when $m \in \Gamma$.

        Returning to the cut-off function $\mu(m,s) = \hat\eta(\frac{m}{\lvert s \rvert^{1/2}})$, we consider three cases.\\
        When $\lvert m \rvert \leq \frac12 \lvert s \rvert^{1/2}$, then $\mu = 1$ and in this case $T_1$ cancels exactly with the $\hat G_1$ contribution and thus, from \eqref{eqn:HD2barG1mu} we obtain
        \[
        \lvert {\rm Div}_1 D_1 D_2 \bar{G}_{1,\mu}(m,s) \rvert \lesssim \lvert m \rvert^{-1/2} \lvert s \rvert^{-5/2},
        \]
        when $\lvert m \rvert \leq \frac12\lvert s \rvert^{1/2}$.
    
        The second case is when $\lvert m \rvert > \lvert s \rvert^{1/2}$, implying that $\mu =0$ and thus $\hat G_{1,\mu} = 0$. In this case \eqref{eqn:HD2barG1mu} becomes 
        \[
        {\rm Div}_1 D_1 D_2 \bar{G}_{1,\mu}(m,s) = -{\rm Div}_1 D_1D_2 \hat{G}_{0}(m,s).
        \]
        It follows from \cite[Lemma~4.10]{2018-antiplanecrack} and Lemma~\ref{lem:Hf-Gamma0} that, for any $m \in \Lambda$,
        \[
        |{\rm Div}_1 D_1D_2 \hat{G}_{0}(m,s)| \lesssim |\nabla^4_m\nabla_s \hat G_0(m,s)| \lesssim |s|^{-3/2}|m|^{-7/2}.
        \]

        That leaves the transition area $\frac{\lvert s \rvert^{1/2}}{2} \leq |m| \leq \lvert s \rvert^{1/2}$, in which we have to additionally consider $-{\rm Div}_1 D_1 D_2 (\mu \hat G_1^m \hat G_1^s)$. We use that, for $i=0,1$,
        \begin{equation}\label{eqn-needed-decay1}
            \lvert D^i \hat G_1^s(s) \rvert \lesssim \lvert s \rvert^{-1/2-i}, \quad \lvert D^i \hat G_1^m(m) \rvert \lesssim \lvert m \rvert^{-1/2-i+\delta},
        \end{equation}
        where the first inequality follows from a direct estimate of the explictly defined $\hat G_1^s$ and the second from Proposition~\ref{prop-ghat1}. We will also use that, due to Lemma~\ref{lem:Hf-Gamma0},
        \begin{equation}\label{eqn-needed-decay2}
        |{\rm Div}D \hat G^m_1(m)| = |{\rm Div}D \omega_2(m)| \lesssim |m|^{-7/2},
        \end{equation}
        and finally that, within this transition area, the cut-off function $\mu$ satisfies, for $i=0,1,2$, 
        \begin{equation}\label{eqn-needed-decay3}
        \lvert D_1^i \mu(m,s) \rvert \lesssim \lvert s \rvert^{-i/2}, \quad \lvert D_1^i D_2 \mu(m,s)\rvert \lesssim \lvert s \rvert^{-i/2-1}.
        \end{equation}
    
        To expand out $-{\rm Div}_1 D_1 D_{2\sigma}[\mu \hat G^m_1 \hat G^s_1](m,s)$ in full, we first recall that 
        \[
            D[uv](m) = \left\{(Du(m))_\rho v(m+\rho) + u(m)(Dv(m))_{\rho} \right\}_{\rho \in \Rc}
        \]
        and that the discrete divergence operator from \eqref{eqn-def-divergence} is given by
        \[
            -{\rm Div}\,g(m) = \sum_{\rho \in \Rc} g_{\rho}(m-\rho) - g_{\rho}(m).
        \]
       By adding and subtracting the same quantity, we observe that 
        \begin{align*}
            -&{\rm Div} D[uv](m) =\\
            &=\sum_{\rho} (Du(m-\rho))_{\rho} v(m) + u(m-\rho)(Dv(m-\rho))_{\rho} - (Du(m))_{\rho}v(m+\rho) - u(m)(Dv(m))_{\rho}\\
            &= \sum_\rho ((Du(m-\rho))_{\rho} - (Du(m))_{\rho})v(m) + (Du(m))_{\rho}v(m) - (Du(m))_{\rho}v(m+\rho)\\
            &\quad\quad\quad + u(m-\rho)(Dv(m-\rho))_{\rho} - u(m)(Dv(m-\rho))_{\rho}  - u(m)((Dv(m))_{\rho} - (Dv(m-\rho))_{\rho})\\
            &= -({\rm Div}Du(m))v(m) + u(m)(-{\rm Div}Dv(m))\\&\quad\quad\quad + \sum_{\rho} \big((Du(m))_{\rho}(Dv(m))_{\rho} + (Du(m-\rho))_{\rho}(Dv(m-\rho))_{\rho}\big).
        \end{align*}
        For notational clarity we now introduce $\sigma \in \Rc(s)$ and instead of  $D_2$ consider $D_{2 \sigma}$. This allows us to rewrite the original expression as 
        \begin{align*}
            -&{\rm Div}_1 D_1 D_{2\sigma}[\mu \hat G^m_1 \hat G^s_1](m,s) = \\ &= -{\rm Div}_1 D_1[\hat G^m_1(m)D_{2\sigma}\mu(m,s) \hat G^s_1(s+\sigma)] -{\rm Div}_1 D_1[\hat G^m_1(m)\mu(m,s)D_{2\sigma}\hat G^s_1(s)]\\
            &= -\hat G^s_1(s+\sigma){\rm Div}_1D_1[\hat G^m_1(m)D_{2 \sigma}\mu(m,s)] - D_{2\sigma}\hat G^s_1(s){\rm Div}_1 D_1[\hat G^m_1(m) \mu(m,s)]\\
            &= -\hat G^s_1(s+\sigma)({\rm Div} D \hat G^m_1(m))D_{2\sigma}\mu(m,s) - \hat G^s_1(s+\sigma)\hat G^m_1(m)({\rm Div}_1 D_1D_{2\sigma}\mu(m,s))\\
            &+ \hat G^s_1(s+\sigma)\sum_{\rho} \big( (D\hat G^m_1(m))_{\rho}(D_1 D_{2\sigma}\mu(m,s))_{\rho} + (D\hat G^m_1(m-\rho))_{\rho}(D_1D_{2\sigma}\mu(m-\rho,s))_{\rho}\big)\\
            &-D_{2\sigma}\hat G^s_1(s)({\rm Div}D\hat G^m_1(m))\mu(m,s) -D_{2\sigma}\hat G^s_1(s)\hat G^m_1(m)({\rm Div}_1D_1\mu(m,s))\\
            &+D_{2\sigma}\hat G^s_1(s)\sum_{\rho}\big( (D\hat G^m_1(m))_\rho(D_1 \mu(m,s))_\rho + (D\hat G^m_1(m-\rho))_\rho(D_1 \mu(m-\rho,s))_\rho\big),
        \end{align*}
        which, crucially, establishes that when two derivatives fall on $\hat G^m_1$, we get $-{\rm Div}D \hat G^m_1$, thus ensuring that the estimate \eqref{eqn-needed-decay2} applies. Combining this with estimates \eqref{eqn-needed-decay1} and \eqref{eqn-needed-decay3}, we thus arrive at 
        \begin{align*}
            |{\rm Div}_1 D_1D_2[\mu \hat G^m_1 \hat G^s_1](m,s)| &\lesssim |s|^{-1/2}|m|^{-7/2}|s|^{-1} + |s|^{-1/2}|m|^{-1/2+\delta}|s|^{-2} + |s|^{-1/2}|m|^{-3/2 + \delta} |s|^{-3/2}\\
            &+|s|^{-3/2}|m|^{-7/2} + |s|^{-3/2} |m|^{-1/2 + \delta} |s|^{-1} + |s|^{-3/2}|m|^{-3/2+\delta} |s|^{-1/2}\\
            &\lesssim |s|^{-5/2}|m|^{-1/2+\delta},
        \end{align*}
        as $\lvert m \rvert$ is comparable to $\lvert s \rvert^{1/2}$.

        \end{proof}
        \begin{remark}\label{rem:exponent-choice}
        The radius $r(s)=\lvert s \rvert^{1/2}$ in the context of the last proof is optimal in terms of balancing the inner and the transition area errors. Note that there always is a third term for $\lvert m\rvert > r(s)$, but we will see that it behaves better and can be discarded.
    
        To see that the exponent $1/2$ is optimal let us look at the above error terms but with $r(s) = \lvert s \rvert^\alpha$, where $\alpha \in (0,1)$. Let us however simplify slightly and only focus on the critical part
        \begin{equation}\label{eqn-critical-sum}
            \sum_{\lvert m \rvert \leq \lvert s \rvert/4} \lvert m \rvert^{-1/2} \lvert {\rm Div}_1 D_1D_2 \bar{G}_{1,\mu}(m,s) \rvert.
        \end{equation}
        This is indeed somewhat of a simplification -- see \eqref{eq:sumforfulloptimization} for the full sum. However this choice is enough to understand the choice of $\alpha$ and the full case behaves the same.
    
        We first note that for the  cut-off $r(s) = |s|^{\alpha}$ we still have 
        \[
        \lvert {\rm Div}_1 D_1 D_2 \bar G_{1,{\mu}}(m,s) \rvert \lesssim \left\{\begin{array}{lr}
            \lvert m \rvert^{-1/2 + \delta}\lvert s \rvert^{-5/2}, & \text{for } \lvert m \rvert \leq \frac12 \lvert s \rvert^{\alpha},\\
            \lvert m \rvert^{-7/2}\lvert s \rvert^{-3/2}, & \text{for } \lvert m \rvert > \lvert s \rvert^{\alpha}.
            \end{array}\right.
      \]
        A lengthy calculation similar to that presented in Lemma~\ref{lem:HD2Gbar1cut}, reveals that in the transition area $\frac12 |s|^{\alpha} \leq |m| \leq |s|^{\alpha}$, we have
        \[
            \lvert {\rm Div}_1 D_1 D_2 \bar G_{1,{\mu}}(m,s) \rvert \lesssim |s|^{-3/2}|m|^{-7/2} + |s|^{-3/2 -\alpha}|m|^{-3/2 + \delta} + |s|^{-3/2-2\alpha}|m|^{-1/2+\delta}.
        \]
        Separating the sum in \eqref{eqn-critical-sum} based on $\mu$ we find
        \begin{align*}
            \sum_{\lvert m \rvert \leq \lvert s \rvert/4}& \lvert m \rvert^{-1/2} \lvert {\rm Div}_1D_1D_2 \bar{G}_{1,{\mu}}(m,s) \rvert \lesssim \\ 
            &\lesssim \lvert s \rvert^{-5/2} \sum_{\lvert m \rvert \leq \frac12|s|^{\alpha}} \lvert m \rvert^{-1}\\
            &+\sum_{\frac12 |s|^{\alpha} \leq |m| \leq |s|^{\alpha}}\left( |s|^{-3/2}|m|^{-7/2} + |s|^{-3/2 -\alpha}|m|^{-3/2 + \delta} + |s|^{-3/2-2\alpha}|m|^{-1/2+\delta}\right)\\
            &+ \lvert s \rvert^{-3/2} \sum_{\lvert m \rvert \geq |s|^{\alpha}} \lvert m \rvert^{-4} \\
            &\lesssim \lvert s \rvert^{-5/2 +\alpha} + \lvert s \rvert^{-3/2 -\alpha+\alpha\delta} + \lvert s \rvert^{-3/2 - 2 \alpha} \lesssim \lvert s \rvert^{-5/2 +\alpha} + \lvert s \rvert^{-3/2 -\alpha+\alpha\delta}
        \end{align*}
        Now we optimize in $\alpha$, further simplifying by ignoring $\delta$. We obtain $\alpha = 1/2$ and
        \begin{align*}
            \sum_{\lvert m \rvert \leq \lvert s \rvert/4} \lvert m \rvert^{-1/2} \lvert HD_2 \bar{G}_{1,\mu} \rvert &\lesssim \lvert s \rvert^{-2+\delta/2}.
        \end{align*}
        Note that this optimization is a balance between the inner term where $\mu=1$ and the transition area where $0 < \mu <1$. The outer term always behaves better.
    
    \end{remark}
    
    \subsubsection{Proof of Theorem~\ref{thm-Gbar1-decay}}\label{sec:proof-thm-Gbar1-decay}
    \begin{proof}[Proof of Theorem \ref{thm-Gbar1-decay}]
    Recalling Definition~\ref{def-cutoff-fcts}, take $\eta(m) := \hat \eta \left(\frac{\lvert m \rvert}{c\lvert s \rvert}\right)$ with $c=\frac{1}{4}$ and further let $\lvert \ell \rvert \leq \lvert s \rvert/16$ and set $v(m):= D_\ell G(m,\ell)$. Then, for any $A(s)$ depending only on $s$, we have
    \begin{align*}
        D_\ell D_s \bar{G}_{1,\mu}(\ell,s) &= D_\ell \Big( \eta(\ell)(D_s \bar{G}_{1,\mu}(\ell,s) - A(s)) \Big)\\
        &= \sum_{m\in \La} D_m \Big(  \eta(m)(D_s \bar{G}_{1,\mu}(m,s) - A(s)) \Big) \cdot Dv(m) \\
        &= \sum_{m\in\La} D_m(D_s\bar G_{1,\mu}(m,s) - A(s)) \cdot D(\eta(m)v(m))\\
        &+ \sum_{m\in\La}\left(D\eta(m) \cdot D_m(D_s \bar G_0(m,s) - A(s))\right)v(m)\\
        &+ \sum_{m \in \La}\left(D\eta(m) \cdot Dv(m)\right)\left(D_s \bar G_0(m,s) - A(s)\right)\\
        &=: T_1 + T_2 + T_3,
    \end{align*}
    where the third equality follows from Lemma~\ref{lem:pushing_eta} and we used that $\bar G_{1,\mu} = \bar G_0$ for $\lvert m \rvert >\lvert s \rvert^{1/2}$. This is fine since 
    \[
    {\rm supp} D\eta \subset B_{|s|/4}(0) \setminus B_{|s|/8}(0) =: \mathcal{A}_s
    \] 
    and we are interested in the regime where $|s|$ is large enough. 
    
    We begin by estimating $T_2$ and $T_3$, which are the boundary terms, since $\mathcal{A}_s$ is an annulus scaling with $|s|$, ensuring that $\tfrac{|s|}{8} > |s|^{1/2}$.
    
    We first notice that thus for $m \in {\rm supp} D\eta$ implies that $|m-\ell| \geq \tfrac{|s|}{16} \geq |\ell|$ and thus 
    \[
    m \in \mathcal A_s \implies |D\eta(m)| \lesssim |s|^{-1}, \quad |D_m D_s \bar G_0(m,s)| \lesssim |s|^{-3+\delta}.
    \]
    Starting with $T_2$, we also note that $D_m A(s) =0$ and that we have
    \[|v(m)| \lesssim \lvert \ell \rvert^{-1/2} \lvert s \rvert^{-1/2+\delta/2}\]
    according to Lemma \ref{lem:DG_decay}. We can thus estimate
    \begin{align*}
        |T_2| &\lesssim \sum_{m \in \mathcal A_s} |s|^{-1} |s|^{-3 + \delta} \lvert \ell \rvert^{-1/2} \lvert s \rvert^{-1/2+\delta/2}\\
        &\lesssim \lvert \ell \rvert^{-1/2} \lvert s \rvert^{2-1-3+\delta-1/2+\delta/2}\\
        &=\lvert \ell \rvert^{-1/2} \lvert s \rvert^{-5/2+3 \delta/2}.
    \end{align*}
    
    For $T_3$, given by 
    \[
    T_3 = \sum_{m \in \La}\left(D\eta(m) \cdot Dv(m)\right)\left(D_s \bar G_0(m,s) - A(s)\right),
    \]
    we use a Cauchy-Schwarz inequality to obtain
    \begin{align*}
        |T_3| &\leq \left(\sum_{m \in \La} |D\eta(m) \cdot Dv(m)|^2\right)^{1/2}\|D_s\bar G_0(\cdot,s)-A(s)\|_{\ell^2(\mathcal A_s)}\\
        &\lesssim (|s|^2|s|^{-2}|s|^{-3+2\delta} |\ell|^{-1})^{1/2}|s|\|D_mD_s\bar G_0(\cdot,s)\|_{\ell^2(\mathcal A_s)} = |s|^{-1/2+\delta}|\ell|^{-1/2}\|D_mD_s\bar G_0(\cdot,s)\|_{\ell^2(\mathcal A_s)},
    \end{align*}
    where we used the Poincare inequality from Lemma \ref{lem:discretePoincare} with $A(s) = (D_s\bar G_0(\cdot,s))_{\mathcal A_s}$. We suppressed the slight increase in annulus width in the notation here as it does not change the estimate.
    And, since
    \[
    \|D_mD_s\bar G_0(\cdot,s)\|_{\ell^2(\mathcal A_s)} \lesssim |s|^{(-6 + 2\delta + 2)/2} = |s|^{-2 + \delta},
    \]
    we obtain
    \[
    |T_3| \lesssim |s|^{-5/2+2\delta}|\ell|^{-1/2}.
    \]
    
    It thus remains to estimate the near-crack-tip term $T_1$. We begin by using $D_m A(s) =0$ and integrating by parts to obtain
    \begin{align}
    T_1 &= \sum_{m\in\La} D_m(D_s\bar G_{1,\mu}(m,s) - A(s)) \cdot D(\eta(m)v(m)) \nonumber\\
    &= \sum_{m\in\La} -{\rm Div}_m D_m D_s\bar G_{1,\mu}(m,s)) (\eta(m)D_\ell G(m,\ell)). \label{eq:sumforfulloptimization},
    \end{align}
    which sets the scene for applying Lemma~\ref{lem:HD2Gbar1cut}. Indeed, to estimate the sum we split the cases $\lvert m \rvert \leq \lvert s \rvert^{1/2}$ and $\lvert m \rvert > \lvert s \rvert^{1/2}$, denoted by $T_1 = T_{1,1} + T_{1,2}$. Also note that $m$ can be close to $\ell$, so care must be taken there as well.
    
    First,
    \begin{align*}
    \lvert T_{1,1} \rvert &=
       \Big\lvert \sum_{\substack{m \in \La \\ \lvert m \rvert \leq \lvert s \rvert^{1/2}}} {\rm Div}_m D_m  D_s\bar G_{1,\mu}(m,s)) (\eta(m)D_\ell G(m,\ell)) \Big\rvert \\
       &\lesssim \sum_{\substack{m \in \La \\ \lvert m \rvert \leq \lvert s \rvert^{1/2}}} \lvert m \rvert^{-1/2+\delta} \lvert s \rvert^{-5/2} \lvert \ell \rvert^{-1/2} \lvert \omega(\ell)-\omega(m)\rvert^{-1+\delta}.
    \end{align*}
    If $\lvert \ell \rvert \geq 2 \lvert s \rvert^{1/2}$, then $\lvert \omega(\ell)-\omega(m)\rvert \sim \lvert \ell \rvert^{1/2}$ and we find
    \[\lvert T_{1,1} \rvert \lesssim \lvert \ell \rvert^{-1 + \delta/2}
    \lvert s \rvert^{-7/4+\delta/2} \lesssim \lvert \ell \rvert^{-1/2}
    \lvert s \rvert^{-2+\delta}\]
    Otherwise we have $\lvert \ell \rvert < 2 \lvert s \rvert^{1/2}$ but we have to compare $m$ and $\ell$ in more detail. To do so, let us disjointedly split
    \[
    \La = \{m \in \La \mid 4\lvert m \rvert \leq \lvert \ell \rvert\} \cup \{m \in \La \mid  4\lvert \omega(m) - \omega(\ell) \rvert \leq \lvert \ell \rvert^{1/2}\} \cup S_3 =: S_1 \cup S_2 \cup S_3.\]
    It is readily verifiable that, by construction, for $m \in S_1 \cup S_3$, we have $\lvert \omega(\ell)-\omega(m)\rvert \sim \lvert \ell \rvert^{1/2}$.
    
    The situation when $m \in S_2$ is a bit more delicate as $S_2$ is the pre-image of a ball under $\omega$ and does not cross the crack. As
    \[  \omega(\ell)+\omega(m) = 2 \omega(\ell) -(\omega(\ell) - \omega(m)),\]
    with $\lvert \omega(\ell)\rvert = \lvert \ell \rvert^{1/2}$ and $\lvert \omega(\ell)-\omega(m)\rvert \leq \lvert \ell \rvert^{1/2}/4$, we find
    \[\frac{7}{4}\lvert \ell \rvert^{1/2} \leq \lvert \omega(\ell)+\omega(m) \rvert \leq \frac{9}{4} \lvert \ell \rvert^{1/2}. \]
    In particular, we obtain
    \[\lvert \omega(\ell)-\omega(m)\rvert = \frac{\lvert \ell - m\rvert}{\lvert \omega(\ell)+\omega(m)\rvert} \sim \frac{\lvert \ell-m \rvert}{\lvert \ell \rvert^{1/2}}.\]
    While $S_2$ is not a ball itself, it is contained in one, with a radius nicely scaling with $\ell$:
    \[\lvert \ell - m \rvert = \lvert \omega(\ell)-\omega(m)\rvert\lvert \omega(\ell)+\omega(m)\rvert \leq \frac{9}{16} \lvert \ell \rvert.\]
    Finally, for notational convenience we introduce the ball $B := B_{|s|^{\frac12}}(0)$.
    
    Thus prepared, we estimate $T_{1,1}$ as follows.
    \begin{align*}
    \lvert T_{1,1} \rvert &\lesssim \sum_{m \in B} \lvert m \rvert^{-1/2+\delta} \lvert s \rvert^{-5/2} \lvert \ell \rvert^{-1/2} \lvert \omega(\ell)-\omega(m)\rvert^{-1+\delta}\\
    &\lesssim \sum_{m \in S_1 \cap B} \lvert m \rvert^{-1/2+\delta} \lvert s \rvert^{-5/2}\lvert \ell \rvert^{-1+\delta/2}\\
    &+ \sum_{m \in S_2 \cap B} \lvert m \rvert^{-1/2+\delta} \lvert s \rvert^{-5/2}\lvert \ell \rvert^{-\delta/2} \lvert \ell - m \rvert^{-1+\delta}\\
    &+ \sum_{m \in S_3 \cap B} \lvert m \rvert^{-1+2\delta} \lvert s \rvert^{-5/2}\lvert \ell \rvert^{-1/2}\\
    &\lesssim  \lvert s \rvert^{-5/2}\lvert \ell \rvert^{1/2+3\delta/2}\\
    &+ \sum_{\substack{m \in B\\\lvert m - \ell \rvert \leq \frac{9}{16}\lvert \ell \rvert}} \lvert \ell \rvert^{-1/2+\delta/2} \lvert s \rvert^{-5/2} \lvert \ell - m \rvert^{-1+\delta}\\
    &+ \sum_{m \in B} \lvert m \rvert^{-1+2\delta} \lvert s \rvert^{-5/2}\lvert \ell \rvert^{-1/2}\\
    &\lesssim  \lvert s \rvert^{-5/2}\lvert \ell \rvert^{1/2+3\delta/2} + \lvert s \rvert^{-5/2} \lvert \ell \rvert^{1/2 + 3\delta/2}+ \lvert s \rvert^{-2+\delta} \lvert \ell \rvert^{-1/2}\\
    &\lesssim \lvert s \rvert^{-2+\delta} \lvert \ell \rvert^{-1/2}
    \end{align*}
    
    The sum outside the squareroot region, $T_{1,2}$, can be estimated in the same way. We actually get slightly better estimates here though we make no use of that directly. We start with 
    \begin{align*}
       \lvert T_{1,2} \rvert &= \Big\lvert \sum_{\substack{ m \in \La \\ \lvert m \rvert > \lvert s \rvert^{1/2}}} {\rm Div}_1D_1 D_s\bar G_{1,\mu}(m,s)) (\eta(m)D_\ell G(m,\ell)) \Big\rvert \\
       &\lesssim \sum_{m \in B_2} \lvert m \rvert^{-7/2} \lvert s \rvert^{-3/2} \lvert \ell \rvert^{-1/2} \lvert \omega(\ell)-\omega(m)\rvert^{-1+\delta},
    \end{align*}
    where $B_2 := B_{\frac{|s|}{4}}(0) \setminus B_{|s|^{\frac12}}(0)$, which follows from the support of the cut-off function $\eta$.
    
    Again, we can first consider the case where $\ell$ is bounded away from the region, $2\lvert \ell \rvert \leq \lvert s \rvert^{1/2}$. In that case
    $\lvert \omega(\ell)-\omega(m)\rvert \sim \lvert m \rvert^{1/2}$ and we find
    \begin{align*}
       \lvert T_{1,2} \rvert &\lesssim \sum_{m: \lvert s \rvert^{1/2} < \lvert m \rvert \leq \lvert s \rvert} \lvert m \rvert^{-7/2} \lvert s \rvert^{-3/2} \lvert \ell \rvert^{-1/2} \lvert \omega(\ell)-\omega(m)\rvert^{-1+\delta}\\
       &\lesssim \sum_{m: \lvert s \rvert^{1/2} < \lvert m \rvert} \lvert m \rvert^{-4+\delta/2} \lvert s \rvert^{-3/2} \lvert \ell \rvert^{-1/2}\\
       &\lesssim \lvert s \rvert^{-5/2+\delta/4} \lvert \ell \rvert^{-1/2}\\
       &\lesssim \lvert s \rvert^{-2+\delta} \lvert \ell \rvert^{-1/2}.
    \end{align*}
    Otherwise we have $2\lvert \ell \rvert > \lvert s \rvert^{1/2}$ and again, need to take more care with comparing $m$ and $\ell$.
    
    We again split $\La = S_1 \cup S_2 \cup S_3$ and estimate
    \begin{align*}
       \lvert T_{1,2} \rvert &\lesssim \sum_{m \in B_2} \lvert m \rvert^{-7/2} \lvert s \rvert^{-3/2} \lvert \ell \rvert^{-1/2} \lvert \omega(\ell)-\omega(m)\rvert^{-1+\delta}\\
       &\lesssim \sum_{m \in S_1 \cap B_2} \lvert m \rvert^{-7/2} \lvert s \rvert^{-3/2} \lvert \ell \rvert^{-1 + \delta/2}\\
       &+ \sum_{m \in S_2\cap B_2} \lvert m \rvert^{-7/2} \lvert s \rvert^{-3/2} \lvert \ell \rvert^{-\delta/2} \lvert \ell - m\rvert^{-1+\delta}\\
       &+ \sum_{m \in S_3 \cap B_2} \lvert m \rvert^{-4+\delta} \lvert s \rvert^{-3/2} \lvert \ell \rvert^{-1/2}\\
       &\lesssim \lvert s \rvert^{-9/4} \lvert \ell \rvert^{-1 + \delta/2}
       + \sum_{\substack{m \in B_2\\ \lvert m-\ell \rvert \leq 3 \lvert \ell \rvert/8}} \lvert \ell \rvert^{-7/2-\delta/2} \lvert s \rvert^{-3/2} \lvert \ell - m\rvert^{-1+\delta}
       + \sum_{\substack{m \in B_2\\ \lvert \ell \rvert/4 < \lvert m \rvert }} \lvert m \rvert^{-4+\delta} \lvert s \rvert^{-3/2} \lvert \ell \rvert^{-1/2}\\
       &\lesssim \lvert s \rvert^{-9/4} \lvert \ell \rvert^{-1 + \delta/2}+ \lvert \ell \rvert^{-5/2+\delta/2} \lvert s \rvert^{-3/2} + \lvert \ell \rvert^{-5/2+\delta} \lvert s \rvert^{-3/2}\\
       &\lesssim \lvert s \rvert^{-9/4} \lvert \ell \rvert^{-1 + \delta/2}+ \lvert \ell \rvert^{-3/2+\delta} \lvert s \rvert^{-2}
       \lesssim \lvert s \rvert^{-2+\delta} \lvert \ell \rvert^{-1/2}.
    \end{align*}
    That concludes all cases for both $T_{1,1}$ and $T_{1,2}$. Overall we have shown
    \[\lvert T_1 \rvert \lesssim \lvert s \rvert^{-2+\delta} \lvert \ell \rvert^{-1/2}, \]
    which is what we set out to prove. 
    \end{proof}
    
    \section{Proofs: atomistic model}\label{sec:proofs-atom-model}
    \subsection{Higher order predictors}\label{sec:proofs-predictors}
    We begin by proving an auxiliary result removing the arbitrarily small $\delta >0$ from the decay estimate for $\bar u_0$ established in \cite{2018-antiplanecrack}.
    \begin{lemma}\label{ubar0-decay}
    It holds that
    \[
    |D\bar u_0(s)| \lesssim |s|^{-3/2}\log|s|.
    \]
    \end{lemma}
    \begin{proof}
    We have
    \begin{equation}\label{eqnDubar0-sum}
    D\bar u_0(s) = \sum_{m \in \La} g_0(m) \cdot Dv(m), 
    \end{equation}
    where $|g_0(m)|\lesssim |m|^{-3/2}$ and $Dv(m) = D_1 D_2 G(m,s)$. We first estimate this sum for $|m| \geq \frac{|s|}{16}$ and split $G = \hat G_0 + \bar G_0$. We know that, for all $m,s \in \La$,
    \[
    |D_1 D_2 \hat G_0(m,s)| \lesssim (1+|\omega(m)||\omega(s)||\omega(m)-\omega(s)|^2)^{-1}
    \]
    from which one can infer (c.f. proof of \cite[Theorem~2.8]{2018-antiplanecrack}) that 
    \[
    \sum_{|m| \geq \frac{|s|}{16}} |g_0(m)\cdot D_1D_2 \hat G_0(m,s)| \lesssim |s|^{-3/2}\log|s|.
    \]
    Turning to $\bar G_0$, by adjusting the constant prefactors defining regions of interest in \cite[Lemma~4.9]{2018-antiplanecrack}, that for $\frac{|s|}{16} \leq |m| \leq \frac{17|s|}{16}$ (crucially including when $m$ is close to $s$), 
    \[
    |D_1 D_2 \bar G_0(m,s)| \lesssim (1+|\omega(m)||\omega(s)||\omega(m)-\omega(s)|^2)^{-1}.
    \]
    We can thus conclude that
    \[
    \sum_{\tfrac{17|s|}{16} \geq |m| \geq \frac{|s|}{16}} |g_0(m)\cdot D_1D_2 \bar G_0(m,s)| \lesssim |s|^{-3/2}\log|s|,
    \]
    by the same argument as for $\hat G_0$.
    
    Finally, by adjusting the constant prefactors in \cite[Lemma 4.12]{2018-antiplanecrack}, we can conclude that 
    \[
    \|D_1 D_2 \bar G_0(\cdot, s)\|_{\ell^2(\Omega_2(s))} \lesssim |s|^{-3/2},
    \]
    where $\Omega_2(s) = \{ m \in \La \mid |m| \geq \frac{17|s|}{16}\}$. This implies that 
    \[
    \sum_{|m| \geq \tfrac{17|s|}{16}} g_0(m)\cdot D_1D_2 \bar G_0(m,s) \lesssim \underbrace{\|g_0\|_{\ell^2(\Omega_2(s))}}_{\lesssim |s|^{-1/2}}\|D_1 D_2 \bar G_0(\cdot, s)\|_{\ell^2(\Omega_2(s))} \lesssim |s|^{-2}.
    \]
    It thus remains to show the estimate
    \[
    \sum_{|m| \leq \frac{|s|}{16}} g_0(m) \cdot Dv(m) = \underbrace{\sum_{|m| \leq \frac{|s|}{16}} g_0(m) \cdot D_1D_2(\hat G_0 + \hat G_{1,\mu})(m,s)}_{:=S_1} + \underbrace{\sum_{|m| \leq \frac{|s|}{16}}g_0(m) \cdot D_1 D_2 \bar G_{1,\mu}(m,s)}_{=: S_2}.
    \]
    Using Theorem~\ref{thm-Gbar1-decay}, we can conclude that 
    \[
    S_2 \lesssim |s|^{-2+\delta}\log|s|.
    \]
    
    To estimate $S_1$, we Taylor expand $\hat G_0$ for $\lvert m \rvert \leq \lvert s \rvert/16$ to obtain
    \begin{align*}
        D_1D_2\hat G_0(m,s) &= \underbrace{D_1D_2\left(-(\nabla F(\omega(s)) + \nabla F(\omega^*(s))\omega(m)\right)}_{D_1D_2\hat G_1^s(s)\omega_2(m)}\\ 
        &+D_1D_2\underbrace{\left(\int_0^1 (1-t)(\nabla^2F(\omega(s) - t(\omega(m)) + \nabla^2F(\omega^*(s) - t\omega(m)))[\omega(m)]^2 dt\right)}_{=:\alpha(m,s)}.
    \end{align*}
    where $|D_1D_2\alpha(m,s)| = O(\lvert s \rvert^{-2})$ uniformly in $m$. As a result
    \[
    \sum_{|m| \leq \frac{|s|}{16}} |g_0(m) \cdot D_1D_2 \alpha(m,s)| \lesssim |s|^{-3/2}. 
    \]
    Similarly, 
    \[
    \sum_{|m| \leq \frac{|s|}{16}} |g_0(m) \cdot D_1D_2 \hat G_1^s(s)(\omega_2(m) + \hat G_1^m(m))| \lesssim |s|^{-3/2},
    \]
    thus it remains to show that the introduction of the cut-off $\mu$ does not affect the overall result. We find that 
    
    \begin{align*}
        \lvert D \bar u_0(s) \rvert  &\lesssim \lvert s \rvert^{-3/2}\log|s| + \Big\lvert \sum_{\lvert m \rvert \leq \lvert s \rvert/16} g_0(m) D_1D_2 \big( (1-\mu(m,s))\hat G_1^m(m) \hat G_1^s(s)\big) \Big\rvert.
    \end{align*}
    For the remaining term we expand out the discrete derivatives as in Lemma~\ref{lem:pushing_eta} and observe that
    \begin{align}
    D_{1\tau}D_{2\sigma}((1-\mu)(m,s) + \hat G^m_1(m)\hat G^s_1(s)) &= (D_{1\tau}D_{2\sigma}(1-\mu)(m,s)) \hat G^s_1(s) \hat G^m_1(m+\tau) \nonumber\\
    &+ (D_{2\sigma}(1-\mu)(m,s))\hat G^s_1(s)(D_{\tau}\hat G^m_1(m))\nonumber\\
    &+ (D_{2\sigma}\hat G^s_1(s))(D_{1\tau}(1-\mu)(m,s)) \hat G_1^m(m+\tau)\nonumber\\
    &+ (D_{\sigma}\hat G^s_1(s))(1-\mu)(m,s)(D_{\tau}\hat G_1^m(m)) \nonumber\\
    &=: T_1 + T_2 + T_3 + T_4. \label{eqn-prod-rule-D1D2}
    \end{align}
    We note that terms $T_1,T_2,T_3$ are only non-zero for $\frac12 |s|^{1/2} \leq |m| \leq |s|^{1/2}$ since at least one derivative falls on the cut-off function $\mu$. Using the known decay of all the functions (c.f. \eqref{eqn-needed-decay1},\eqref{eqn-needed-decay2}, \eqref{eqn-needed-decay3}), we can thus estimate
    \[
    \sum_{|m| \leq \frac{|s|}{16}} |g_0(m)||T_1 + T_2 + T_3| \lesssim \sum_{\frac12 |s|^{1/2} \leq |m| \leq |s|^{1/2}} |m|^{-3/2}(|s|^{-2}|m|^{-1/2 + \delta} + |s|^{-3/2}|m|^{-3/2 + \delta}) \lesssim |s|^{-2+\delta/2}. 
    \]
    Finally, since $(1-\mu)$ is only non-zero for $|m| \geq \frac12 |s|^{1/2}$, we also have
    \[
    \sum_{|m| \leq \frac{|s|}{16}} |g_0(m)||T_4| \lesssim |s|^{-3/2} \sum_{\frac12 |s|^{1/2} \leq |m| \leq \frac{|s|}{16}}|m|^{-3+\delta} \lesssim |s|^{-2 + \delta/2}.
    \]
    It thus follows that
    \begin{align*}
        \Big\lvert &\sum_{\lvert m \rvert \leq \lvert s \rvert/16} g_0(m) D_1D_2 \big( (1-\mu(m,s))\hat G_1^m(m) \hat G_1^s(s)\big) \Big\rvert \lesssim |s|^{-2 + \delta/2}.
    \end{align*}
    We have thus estimated all the terms involved in \eqref{eqnDubar0-sum} and we are able to conclude that indeed
    \[
    |D\bar u_0(s)| \lesssim |s|^{-3/2}\log|s|,
    \]
    which is what we set out to prove. 
    \end{proof}
    
    We next describe the next order predictor, $\hat u_1$, which was given in \eqref{eqn-hatu1-formula}. As will soon become clear, we require that $\hat u_1$ satisfies the following PDE.
    \begin{alignat}{2}
    -\Delta \hat u_1 &= \divo H \quad &&\text{in}\quad \R^2\setminus \Gamma_0\nonumber \\
        \nabla \hat u_1 \cdot \nu &= 0,\quad &&\text{ on }\quad \Gamma_0,\label{eq:u1PDE}
    \end{alignat}
    where
    \[
    H = \frac{1}{6} \phi^{(iv)}(0)\begin{pmatrix}
        (\partial_1 \hat u_0)^3 \\ (\partial_2 \hat u_0)^3 \end{pmatrix},\quad  \divo H = \frac{1}{2} \phi^{(iv)}(0) \Big((\partial_1 \hat u_0)^2\partial^2_{11}\hat u_0 + (\partial_2 \hat u_0)^2\partial^2_{22}\hat u_0 \Big).
    \]
    We prove the following. 
    \begin{lemma}
        The equation \ref{eq:u1PDE} admits the explicit solution 
    \begin{equation}\label{eqn: uhat1}
        \hat u_1(x) := -\frac{1}{64} \phi^{(iv)}(0) K^3 r_x^{-1/2}\Big( \log r_x \sin \tfrac{\theta_x}{2} + \frac{1}{6} \sin\tfrac{5\theta_x}{2}\Big).
    \end{equation}
    \end{lemma}
    \begin{remark}
        We do not claim uniqueness. In fact, we will look at the homogeneous part below with $\hat u_2$.
    \end{remark}
    \begin{proof}
        Remember that $\hat u_0 = K \omega_2(x)$, where $\omega$ is the complex square-root. Or in polar coordinates, $\hat u_0(r \cos \theta, r \sin \theta) = \hat v_0(r, \theta) := K r^{1/2} \sin \tfrac{\theta}{2}$. For the derivatives we get
        \begin{align*}
            \partial_1 \hat u_0 &= \partial_r \hat v_0 \cos \theta - \frac{1}{r} \partial_\theta \hat v_0 \sin \theta \\
            &= \frac{K}{2} r^{-1/2} \Big( \sin \tfrac{\theta}{2} \cos \theta - \cos \tfrac{\theta}{2} \sin \theta \Big)\\
            &= -\frac{K}{2} r^{-1/2}\sin \tfrac{\theta}{2}
            \end{align*}
            \begin{align*}
            \partial_2 \hat u_0 &= \partial_r \hat v_0 \sin \theta + \frac{1}{r} \partial_\theta \hat v_0 \cos \theta \\
            &= \frac{K}{2} r^{-1/2} \Big( \sin \tfrac{\theta}{2} \sin \theta + \cos \tfrac{\theta}{2} \cos \theta  \Big)\\
            &= \frac{K}{2} r^{-1/2}\cos\tfrac{\theta}{2}.
        \end{align*}
        
        Let us use that to express the entire non-linearity in polar coordinates.
    
        First, let use define
        \begin{align*}
            G_r(r, \theta) &= H(r \cos \theta, r \sin \theta) \cdot \begin{pmatrix}
        \cos \theta \\ \sin \theta \end{pmatrix}\\
        G_\theta(r, \theta) &= H(r \cos \theta, r \sin \theta) \cdot \begin{pmatrix}
        -\sin \theta \\ \cos \theta \end{pmatrix}
        \end{align*}
    
        Calculating the derivatives we find the (standard) formula for the divergence in polar coordinates
        \begin{equation} \label{eq:polardiv}
            \divo H(x) = \partial_r G_r (r_x, \theta_x) + \frac{1}{r} \partial_\theta G_\theta (r_x, \theta_x) + \frac{1}{r} G_r (r_x, \theta_x).
        \end{equation}
        In our case,
        \begin{align*}
            G_r(r, \theta) &= \frac{1}{6} \phi^{(iv)}(0) \Big( (\partial_1 \hat u_0)^3 \cos \theta + (\partial_2 \hat u_0)^3 \sin \theta \Big) \\
            &= \frac{1}{48} \phi^{(iv)}(0) K^3 r^{-3/2}\Big( -\sin^3 \tfrac{\theta}{2} \cos \theta + \cos^3 \tfrac{\theta}{2} \sin \theta \Big) \\
            &= \frac{1}{192} \phi^{(iv)}(0) K^3 r^{-3/2} \Big( 3 \sin \tfrac{\theta}{2} + \sin \tfrac{5\theta}{2}\Big)
            \end{align*}
            and        
        \begin{align*}
        G_\theta(r, \theta) &= \frac{1}{6} \phi^{(iv)}(0) \Big( -(\partial_1 \hat u_0)^3 \sin \theta + (\partial_2 \hat u_0)^3 \cos \theta \Big)\\
        &= \frac{1}{48} \phi^{(iv)}(0) K^3 r^{-3/2} \Big( \sin^3 \tfrac{\theta}{2} \sin \theta + \cos^3 \tfrac{\theta}{2} \cos \theta \Big)\\
        &= \frac{1}{192} \phi^{(iv)}(0) K^3 r^{-3/2} \Big( 3 \cos \tfrac{\theta}{2} + \cos \tfrac{5\theta}{2}\Big).
        \end{align*}
        Now we can insert these expressions into \eqref{eq:polardiv} to obtain
        \begin{align*}
            \divo H(r \cos \theta, r \sin \theta) &= \partial_r G_r (r, \theta) + \frac{1}{r} \partial_\theta G_\theta (r, \theta) + \frac{1}{r} G_r (r, \theta)\\
            &= \frac{1}{192} \phi^{(iv)}(0) K^3 r^{-5/2} \Big( (1-3/2)(3 \sin \tfrac{\theta}{2} + \sin \tfrac{5\theta}{2}) + (-3/2 \sin \tfrac{\theta}{2} - 5/2 \sin \tfrac{5\theta}{2})\Big)\\
            &= -\frac{1}{64} \phi^{(iv)}(0) K^3 r^{-5/2} ( \sin \tfrac{\theta}{2} + \sin \tfrac{5\theta}{2} ).
        \end{align*}
        Let us also remember that the Laplace in polar coordinates is given as
        \[\Delta \hat u_1 = \partial_r^2 \hat v_1 + \frac{1}{r} \partial_r \hat v_1 + \frac{1}{r^2} \partial_\theta^2 \hat v_1, \]
        where we now also write $\hat u_1(r \cos \theta, r \sin \theta) = \hat v_1(r, \theta)$.
        Overall, equation \eqref{eq:u1PDE} becomes
        \begin{alignat}{2}
            -\partial_r^2 \hat v_1 - \frac{1}{r} \partial_r \hat v_1 - \frac{1}{r^2} \partial_\theta^2 \hat v_1 &= -\frac{1}{64} \phi^{(iv)}(0) K^3 r^{-5/2} ( \sin \tfrac{\theta}{2} + \sin \tfrac{5\theta}{2}) \quad &&\text{in}\quad (0,\infty) \times (-\pi, \pi) \nonumber \\
            \partial_\theta \hat v_1 &= 0,\quad &&\text{ for }\quad \theta = \pm \pi,\label{eq:u1PDEpolar}
        \end{alignat}
        in polar coordinates.
        Now we can insert our function
        $\hat v_1 = c_1 r^{-1/2} \log r \sin \tfrac{\theta}{2} + c_2 r^{-1/2} \sin \tfrac{5\theta}{2}$.
        As
        \[(r^{-1/2} \log r)' = r^{-3/2} (1-1/2 \log r)\]
        and 
        \[(r^{-1/2} \log r)'' = r^{-5/2} (-2+3/4 \log r),\]
        for the $\sin \tfrac{\theta}{2}$ coefficient we need to satisfy
        \[c_1 (2-3/4 \log r - 1 + 1/2 \log r + 1/4 \log r)r^{-5/2} = -\frac{1}{64} \phi^{(iv)}(0) K^3 r^{-5/2}, \]
        which in fact just simplifies to
        \[c_1 = -\frac{1}{64} \phi^{(iv)}(0) K^3. \]
        For $\sin \tfrac{5\theta}{2}$ we get
        \[c_2(-3/4 + 1/2 + 25/4 )r^{-5/2} = -\frac{1}{64} \phi^{(iv)}(0) K^3 r^{-5/2}\]
        or
        \[c_2 = -\frac{1}{6} \frac{1}{64} \phi^{(iv)}(0) K^3.\]
    \end{proof}
    
    The next lemma concerns the effect that the introduction of $\hat u_1$ has on $\sum_{m \in \La} D\bar u_1(m)\cdot Dv(m)$ for $v \in \Hc$, as compared to the known equality for $\bar u_0$, given by 
    \[
    \sum_{m \in \La}D\bar u_0(m) \cdot Dv(m) = \sum_{m \in \La} g_0(m) \cdot Dv(m),
    \]
    where $|g_0(m)| \lesssim |m|^{-3/2}$. The following is true instead for $\bar u_1$. 
    \begin{lemma} \label{lem:stressdecay}
        There is a $g_1 \in \ell^2(\La; \R^2)$ such that for all $v \in \Hc$ it holds that
        \[
        \sum_{m \in \La} D\bar u_1(m) \cdot Dv(m)= \sum_{m \in \La}g_1(m)\cdot Dv(m).
        \]
        Furthermore, $|g_1(m)| \lesssim |m|^{-5/2}\log|m|$.
    \end{lemma}
    \begin{proof}
    Recall that $u = \hat u_0 + \hat u_1 + \bar u_1$. Consider
    \begin{align*}
        \sum_{\Lambda} D\bar u_1 \cdot Dv &= -\sum_{m \in \Lambda} D\hat u_0(m) \cdot Dv(m) \\
        &\quad -\sum_{m \in \Lambda} \sum_{\rho \in \Rc(m)} \big( D_{\rho}\hat u_1(m) + \frac{1}{6}\phi^{(iv)}(0)(D_{\rho} \hat u_0(m))^3 \big) D_{\rho}v(m) \\
        &\quad -\sum_{m \in \Lambda} \sum_{\rho \in \Rc(m)} \Big(\phi'(D_{\rho}u(m)) - D_{\rho}u(m) - \frac{1}{6}\phi^{(iv)}(0)(D_{\rho}\hat u_0(m))^3  \Big) D_{\rho}v(m)\\
        &\quad =: T_1 + T_2 + T_3
    \end{align*}
    for any $v \in \Hc$. In the following we will show that, for any $i \in \{1,2,3\}$,
    \[
    T_i = \sum_{m \in \La}f_i(m)\cdot Dv(m),
    \]
    for some $f_i$, as well as the estimates $\lvert f_i(m) \rvert \lesssim |m|^{-5/2}$, for $i=1,2$,  and $\lvert f_3(m) \rvert \lesssim |m|^{-5/2}\log|m|$.
    
    The assertion holds for $T_3$ thanks to Lemma~\ref{ubar0-decay} and a direct Taylor expansion of terms in the bracket. Similarly, the assertion for $T_1$ holds due to Lemma~\ref{lem:symmetric_I_decay}. We now apply a similar argument to estimate $T_2$.
    
    Let us start by rewriting the PDE that $\hat u_1$ satisfies, given by \eqref{eq:u1PDE}, in weak form by testing with $Iv$, where $I$ is the standard interpolation of lattice functions,  as discussed in Section~\ref{sec:interpol}. For compactly supported $v$ we obtain
    \[
    J := 2 \int_{\R^2\setminus \Gamma_0} \nabla \hat u_1 \cdot \nabla Iv + H \cdot \nabla Iv\, dx = 0,
    \]
    without any boundary term due to the boundary condition for $\hat u_1$ and the fact that $H \big|_{\theta = \pm \pi} = 0$, which follows from the boundary condition that $\hat u_0$ satisfies. We multiplied by $C_{\Lambda}=2$ for direct comparison with the lattice sums where each part of the domain will be used once for each direction of the bond.
    
    As discussed in Section~\ref{sec:interpol}, we can rewrite $J$ as 
    \[
    J = \frac12 \sum_{m \in \La}\sum_{\rho \in \Rc(m)}\left(\int_{R_{m,\rho}}( \nabla_{\rho}\hat u_1 + H \cdot \rho)dx\right)D_{\rho}v(m).
    \]
    Subtracting $J=0$ from $T_2$, we get
    \begin{align*}
    T_2 = &-\sum_{m \in \La}\sum_{\rho \in \Rc(m)}\left( D_{\rho}\hat u_1(m) - \int_{R_{m,\rho}}\nabla_\rho \hat u_1(x)dx \right)D_{\rho}v(m)\\
    &-\sum_{m \in \La}\sum_{\rho \in \Rc(m)}\left(\frac16\phi^{(iv)}(0)(D_{\rho}\hat u_0(m))^3 - \int_{R_{m,\rho}}\frac{1}{6} \phi^{(iv)}(0) (\nabla_\rho \hat u_0(x))^3 dx \right)D_{\rho}v(m)\\
    =:& T_{2,1} + T_{2,2},
    \end{align*}
    where we have used the fact that we can write $H(x) \cdot \rho = \frac{1}{6} \phi^{(iv)}(0) (\partial_\rho \hat u_0(x))^3$. By retracing the argument in Section~\ref{sec:interpol}, we thus arrive at
    \[
    T_{2,1} = \sum_{m\in \La} \hat h_1(m) \cdot Dv(m),
    \]
    where $\hat h_1 \in \ell^2(\La; \R^4)$ is given by
    \begin{align*}
        (\hat h_1)_{\rho}(m) = \begin{cases}
            D_{\rho}\hat u_1(m) - \int_{R_{m, \rho}} \nabla_{\rho}\hat u_1(x)dx, \quad &b(m,\rho) \not\subset Q_0,\\
            D_{\rho}\hat u_1(m) - C_{m+\rho,m}, \quad &b(m,\rho) \subset Q_0.
        \end{cases}
    \end{align*}
    where the formulae for the finite $C_{m+\rho,m}$ can be readily obtained, similarly to the argument in \cite[Section~4.2.1.]{2018-antiplanecrack}. By construction, we also have that
    \begin{subequations}
        \begin{align*}
        b(m,\rho) \not\subset \Gamma &\implies |(\hat h_1(m))_{\rho}| \lesssim |\nabla^3 \hat u_1(m)| \lesssim |m|^{-7/2}\log|m|\\
        b(m,\rho) \subset \Gamma\setminus Q_0 &\implies |(\hat h_1(m))_{\rho}| \lesssim |\nabla^2 \hat u_1(m)| \lesssim |m|^{-5/2}\log|m|.
    \end{align*}
    \end{subequations}
    Likewise, a Taylor expansion of $D_{\rho}\hat u_0(m)$ and a first order quadrature estimate we get
    \[
    T_{2,2} = \sum_{m \in \La} k(m) \cdot Dv(m),
    \]
    where $|k(m)| \lesssim |m|^{-5/2}$. This allows us to establish that
    \[
    T_2 = \sum_{m \in \La}\hat f_2(m) \cdot Dv(m),\quad |\hat f_2(m)| \lesssim |m|^{-5/2}\log|m|.
    \]
    It remains to invoke the symmetrisation trick from Lemma~\ref{lem:symmetric_I_decay}, which concerns replacing $Iv$ with $I^{\rm sym}v$, to conclude that
    \[
    T_2 = \sum_{m \in \La} f_2(m) \cdot Dv(m), \quad |f_2(m)| \lesssim |m|^{-5/2}.
    \]
    The result of the lemma has thus been established.
    
    \end{proof}
    \subsection{Proof of Theorem~\ref{thm:ubar2_decay}}\label{sec:proof-thm:ubar2_decay}
    We are now in a position to prove the main decay result for $\bar u_2$. We note that the steps of the argument are similar to the proof of Lemma~\ref{ubar0-decay}, but with two big differences. First is the improved decay of $g_1$ from Lemma~\ref{lem:stressdecay} and the second is that we first focus on $\bar u_1$, which lead us to defining $\hat u_2$ as needed to obtain the ultimate result. 
    \begin{proof}[Proof of Theorem~\ref{thm:ubar2_decay}]
    First consider $\bar u_1$. From Lemma \ref{lem:stressdecay} and the solution property of the Green's function we deduce that
    \begin{align*}
        D \bar u_1(s) = \sum_{m \in \La} g_1(m) D_1D_2 G(m,s).
    \end{align*}
    Lemma \ref{lem:stressdecay} also states that
    \[\lvert g_1 \rvert \lesssim \lvert m \rvert^{-5/2} \log \lvert m \rvert.\]
    
    As in Lemma~\ref{ubar0-decay}, we will first estimate this sum for $|m| \geq \frac{|s|}{16}$ and again split $G = \hat G_0 + \bar G_0$. In particular, retracing the steps in Lemma~\ref{ubar0-decay}, but using the improved decay of $g_1$ when compared to $g_0$, we immediately can conclude that
    \[
    \sum_{|m| \geq \frac{|s|}{16}} |g_1(m)\cdot D_1D_2 \hat G_0(m,s)| \lesssim |s|^{-5/2}\log|s|.
    \]
    \[
    \sum_{\tfrac{17|s|}{16} \geq |m| \geq \frac{|s|}{16}} |g_1(m)\cdot D_1D_2 \bar G_0(m,s)| \lesssim |s|^{-5/2}\log|s|,
    \]
    and 
    \[
    \sum_{|m| \geq \tfrac{17|s|}{16}} g_1(m)\cdot D_1D_2 \bar G_0(m,s) \lesssim \underbrace{\|g_1\|_{\ell^2(\Omega_2(s))}}_{\lesssim |s|^{-3/2}}\|D_1 D_2 \bar G_0(\cdot, s)\|_{\ell^2(\Omega_2(s))} \lesssim |s|^{-3}.
    \]
    Therefore,
    \begin{align*}
        D \bar u_1(s)  &=\sum_{\lvert m \rvert \leq \lvert s \rvert/16} g_1(m)\cdot D_1D_2 G(m,s) +  O(\lvert s \rvert^{-5/2} \log \lvert s \rvert)\\
        &= \sum_{\lvert m \rvert \leq \lvert s \rvert/16} g_1(m)\cdot D_1D_2 (\hat G_0 + \hat G_{1,\mu} )(m,s)  +  O(\lvert s \rvert^{-2+\delta})
    \end{align*}
    where we used Theorem~\ref{thm-Gbar1-decay} for the second line, since
    \[
    \sum_{|m| \leq |s|/16} |g_1(m)||D_1D_2(\bar G_{1,\mu})(m,s)| \lesssim |s|^{-2+\delta} \sum_{|m| \leq |s|/16} |m|^{-3}\log|m| \lesssim |s|^{-2+\delta}.
    \]
    
    A Taylor expansion of $\hat G_0$ for $\lvert m \rvert \leq \lvert s \rvert/16$ reveals that
    \begin{align*}
        D_1D_2\hat G_0(m,s) &= \underbrace{D_1D_2\left(-(\nabla F(\omega(s)) + \nabla F(\omega^*(s))\omega(m)\right)}_{D_1D_2\hat G_1^s(s)\omega_2(m)}\\ 
        &+\underbrace{D_1D_2\left(\int_0^1 (1-t)(\nabla^2F(\omega(s) - t(\omega(m)) + \nabla^2F(\omega^*(s) - t\omega(m)))[\omega(m)]^2 dt\right)}_{\mathcal{O}(|s|^{-2})}.
    \end{align*}
    where the $O(\lvert s \rvert^{-2})$ is uniformly in $m$. 
    
    This implies that, up to the cut-off function $\mu$, the only terms not of order $\mathcal{O}(|s|^{-2+\delta})$ are given by 
    \[
    \sum_{|m| \leq \frac{|s|}{16}} g_1(m) \cdot D_1D_2 \hat G_1^s(s)(\omega_2(m) + \hat G_1^m(m)),
    \]
    which motivates setting 
    \begin{equation}\label{def-hatu_2_proofs}
    \hat u_2(s):= -\hat G_1^s(s) \sum_m g_1(m) (D\hat G_1^m(m) + D \omega_2(m)).
    \end{equation}
    Then we can combine the arguments and see what remains. We find
    \begin{align}\label{eqn-dbaru2-est}
        \lvert D \bar u_2(s) \rvert  &\lesssim \lvert s \rvert^{-2+\delta} + \Big\lvert \sum_{\lvert m \rvert \leq \lvert s \rvert/16} g_1(m) D_1D_2 \big( (1-\mu(m,s))\hat G_1^m(m) \hat G_1^s(s)\big) \Big\rvert.
    \end{align}
    Using \eqref{eqn-prod-rule-D1D2} and the known decay of all the functions (c.f.~\eqref{eqn-needed-decay1}, \eqref{eqn-needed-decay2}, \eqref{eqn-needed-decay3}), we can estimate the remaining term as
    \begin{align*}
    \sum_{|m| \leq \frac{|s|}{16}} |g_1(m)||T_1 + T_2 + T_3| &\lesssim \sum_{\frac12 |s|^{1/2} \leq |m| \leq |s|^{1/2}} |m|^{-5/2}\log|m|(|s|^{-2}|m|^{-1/2 + \delta} + |s|^{-3/2}|m|^{-3/2 + \delta})\\ 
    &\lesssim |s|^{-5/2+\delta/2}\log|s|. 
    \end{align*}
    Finally, since $(1-\mu)$ is only non-zero for $|m| \geq \frac12 |s|^{1/2}$, we also have
    \[
    \sum_{|m| \leq \frac{|s|}{16}} |g_1(m)||T_4| \lesssim |s|^{-3/2} \sum_{\frac12 |s|^{1/2} \leq |m| \leq \frac{|s|}{16}}|m|^{-4+\delta}\log|m| \lesssim |s|^{-5/2 + \delta/2}\log|s|.
    \]
    It thus follows that
    \begin{align*}
        \Big\lvert &\sum_{\lvert m \rvert \leq \lvert s \rvert/16} g_1(m) D_1D_2 \big( (1-\mu(m,s))\hat G_1^m(m) \hat G_1^s(s)\big) \Big\rvert \lesssim |s|^{-5/2 + \delta/2}\log|s|.
    \end{align*}
    Comparing with \eqref{eqn-dbaru2-est}, we have thus established that 
    \[
    |D\bar u_2(s)|\lesssim |s|^{-2+\delta},
    \]
    which concludes the proof. 

    \end{proof}
    
    \bibliographystyle{siam_no_dash}
    \bibliography{papers}

    \end{document}